\documentclass[review,onefignum,onetabnum]{siamart190516}
\usepackage{subcaption} 
\usepackage{algorithmic}
\newcommand{\bI}{\mathbf{I}}

\newcommand{\ba}{\mathbf{a}}
\newcommand{\bb}{\mathbf{b}}
\newcommand{\br}{\mathbf{r}}
\newcommand{\bs}{\mathbf{s}}

\newcommand{\bG}{\mathbf{G}}
\newcommand{\bK}{\mathbf{K}}

\newcommand{\Knys}{{\bK}_{nys}}
\newcommand{\bx}{\mathbf{x}}
\newcommand{\by}{\mathbf{y}}

\newcommand{\bL}{\mathbf{L}}
\newcommand{\bU}{\mathbf{U}}
\newcommand{\norm}[1]{\Vert #1 \Vert}
\newcommand{\Tr}[1]{\operatorname{Trace}(#1)}
\newcommand{\matern}{Mat\'ern-3/2 }

\newcommand{\tn}[1]{\begin{tiny{#1}\end{tiny}}}

\newcommand{\shifanrewriting}[1]{{#1}}
\DeclareMathOperator*{\argmax}{arg\,max}

\newcommand{\noise}{\mu\,\bI}
\newcommand{\Schur}{\bK_{22} + \noise -\bK_{12}^{\top}(\bK_{11}+\noise)^{-1}\bK_{12}}
\newcommand{\bo}{\mathbf{0}}
\newcommand{\bM}{\mathbf{M}}
\def\AFN{\texttt{AFN}}
\def\RAN{\texttt{RAN}}

\newtheorem{manualtheoreminner}{Theorem}
\newenvironment{manualtheorem}[1]{%
  \renewcommand\themanualtheoreminner{#1}%
  \manualtheoreminner
}{\endmanualtheoreminner}

\usepackage{lipsum}
\usepackage{amsfonts}
\usepackage{graphicx}
\usepackage{epstopdf}
\usepackage{algorithmic}
\ifpdf
  \DeclareGraphicsExtensions{.eps,.pdf,.png,.jpg}
\else
  \DeclareGraphicsExtensions{.eps}
\fi

\newsiamremark{remark}{Remark}
\newsiamremark{hypothesis}{Hypothesis}
\crefname{hypothesis}{Hypothesis}{Hypotheses}
\newsiamthm{claim}{Claim}

\headers{Multilevel preconditioner for Gaussian kernel}{xxx}

\title{Robust structured block preconditioner for Gaussian Kernel Matrices\thanks{Submitted to the editors DATE.
\funding{This work was funded by the Fog Research Institute under contract no.~FRI-454.}}}

\author{XXX\thanks{Emory 
  (\email{xxx@emory.edu}, \url{xxx}).}
\and xxx\thanks{Georgia Tech
  (\email{xxx}, \email{xxx}).}
\and xxx\footnotemark[3]}

\usepackage{amsopn}

\makeatletter
\newcommand*{\addFileDependency}[1]{
  \typeout{(#1)}
  \@addtofilelist{#1}
  \IfFileExists{#1}{}{\typeout{No file #1.}}
}
\makeatother

\title{An Adaptive Factorized Nystr\"om Preconditioner for Regularized Kernel Matrices}
\author{Shifan Zhao\thanks{Department of Mathematics, Emory University, Atlanta, GA 30322 (\email{shifan.zhao@emory.edu}, \email{tianshi.xu@emory.edu},\email{yxi26@emory.edu}). The research of S.~Zhao, T.~Xu and Y.~Xi is supported by NSF award OAC 2003720. T.~Xu is also partially supported by NSF award DMS 1912048.}
\and Tianshi Xu$^{\ast}$
\and Hua Huang\thanks{School of Computational Science and Engineering, Georgia Institute of Technology, Atlanta, GA 30332 (\email{huangh223@gatech.edu},\email{echow@cc.gatech.edu}). The research of H.~Huang and E.~Chow is supported by NSF award OAC 2003683.}
\and Edmond Chow$^{\dagger}$ 
\and Yuanzhe Xi$^{\ast}$}
\headers{S. Zhao, T. Xu, H. Huang, E. Chow, and Y. Xi}{Adaptive Factorized Nystr\"om Preconditioner}

\begin{document}

\maketitle
\begin{abstract}
The spectrum of a kernel matrix significantly depends on the parameter
values of the kernel function used to define the kernel matrix.
This makes it challenging to design a preconditioner for a regularized
kernel matrix that is robust across different parameter values.
This paper proposes the Adaptive Factorized Nystr\"om
(\texttt{AFN}) preconditioner.  The preconditioner is designed for
the case where the rank $k$ of the Nystr\"om approximation
is large, i.e., for kernel function parameters that lead to kernel
matrices with eigenvalues that decay slowly. 
\texttt{AFN}  deliberately chooses a well-conditioned submatrix
 to solve with and corrects a Nystr\"om approximation with a factorized sparse approximate  matrix inverse. This makes
\texttt{AFN} efficient for kernel
matrices with large numerical ranks.  \texttt{AFN} also adaptively
chooses the size of this submatrix to balance accuracy and cost.
\end{abstract}


\section{Introduction}
In this paper, we seek efficient preconditioning techniques for the iterative solution of large, regularized linear systems associated with a kernel matrix $\mathbf{K}$,
\begin{equation}
\label{eq:Problem}
(\mathbf{K} + \mu\bI) \, \ba = \bb,
\end{equation}
where $\bI$ is the $n\times n$ identity matrix,  $\mu\in\mathbb{R}$ is a regularization parameter, $\ba, \bb \in \mathbb{R}^{n}$, and $\mathbf{K}\in \mathbb{R}^{n\times n}$ is the kernel matrix whose $(i,j)$-th entry is defined as $\mathcal{K}(\bx_i,\bx_j)$ with a symmetric positive semidefinite (SPSD) kernel function $\mathcal{K}:\mathbb{R}^d\times \mathbb{R}^d \rightarrow \mathbb{R}$ and data points $\{\bx_i\}_{i=1}^{n}$. A function $\mathcal{K}(\cdot,\cdot)$ defined over a domain $\mathcal{X}$ is SPSD if $\mathcal{K}(\bx_i,\bx_j) = \mathcal{K}(\bx_j,\bx_i), ~\forall \bx_i,\bx_j \in \mathcal{X}$ and $\sum_{i=1}^{n}\sum_{j=1}^{n}c_ic_j\mathcal{K}(\bx_i,\bx_j) \ge 0, ~\forall \bx_1, \dots, \bx_n \in \mathcal{X}$  given $n\in \mathbb{N}$ and $c_1,c_2,\dots, c_n \in \mathbb{R}.$
For example, $\mathcal{K}$ can be chosen as a Gaussian kernel function,
\begin{equation}
\mathcal{K}(\bx,\by) = \exp\left(-\frac{\Vert\bx-\by\Vert_2^2}{l^2}  \right),
\label{eq:gaussianf}
\end{equation} 
where $l$ is a kernel function parameter called the \emph{length-scale}.

Linear systems of the form \eqref{eq:Problem} appear in many applications,
including Kernel Ridge Regression (KRR) \cite{NIPS2015_f3f27a32}
and Gaussian Process Regression (GPR) \cite{3569}. 
When the number of data points $n$ is small, solution methods based
on dense matrix factorizations are the most efficient. When $n$ is large,
a common approach is to solve \eqref{eq:Problem} using a sparse
or low-rank approximation to $\mathbf{K}$ \cite{schafer_sparse_2021,schafer_compression_2021,nys2017}. In this paper, we
pursue an exact solution approach for \eqref{eq:Problem} with iterative
methods. Fast matrix-vector multiplications by $\mathbf{K}$ for
the iterative solver are available through fast transforms \cite{fast-gauss,NIPS2004_85353d3b}
and hierarchical matrix methods \cite{7130620,smash,erlandson_accelerating_2020,DDSMASH,sherry19,ambikasaran2013mathcal,chenhan2019distributed,rebrova2018study,si2014memory,march2015robust}.
This paper specifically addresses the problem of preconditioning
for the iterative solver.

In KRR, GPR, and other applications,
the kernel function parameters must be estimated that fit the data at hand.
This involves an optimization process, for example maximizing a
likelihood function, which in turn involves solving
\eqref{eq:Problem} for kernel matrices given the same
data points but different values of the kernel function parameters.
Different values of the kernel function parameters lead to different
characteristics of the kernel matrix.  For the Gaussian kernel function
above, Figure \ref{fig:varyingl} (left) shows the eigenvalue spectrum
of 61 regularized $1000 \times 1000$ kernel matrices. $1000$ data points are sampled inside a cube with edge length $10$. In all the experiments, the side of the $d$-dimensional cube is scaled by $n^{1/d}$ in order to maintain a constant density as we increase the number of data points. %
Figure \ref{fig:varyingl} (right) demonstrates the significant variation in the number of iterations required by the unpreconditioned Conjugate Gradient (CG) method to solve linear systems associated with these kernel matrices. Notably, linear systems are more efficiently solved at the extremes of $l$, whether very high or very low, compared to moderate values. This is consistent with the theoretical convergence results of CG. The convergence of CG is fast, for example, when there exists a cluster of eigenvalues with only a few outliers deviating from this cluster. In the context of kernel matrices, a small $l$ yields a kernel matrix that closely resembles a diagonal matrix which has a tight cluster of eigenvalues. On the other hand, a large $l$ typically results in a scenario where, aside from a few outliers, most eigenvalues are around the regularization parameter $\mu$. 

\begin{figure}[t]
\centering
\includegraphics[width=.48\linewidth]{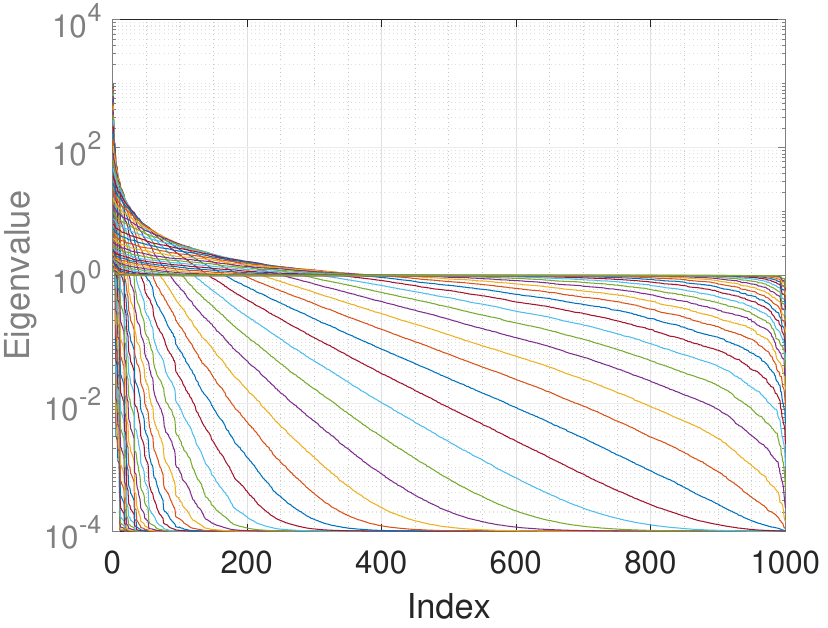}
\includegraphics[width=.475\linewidth]{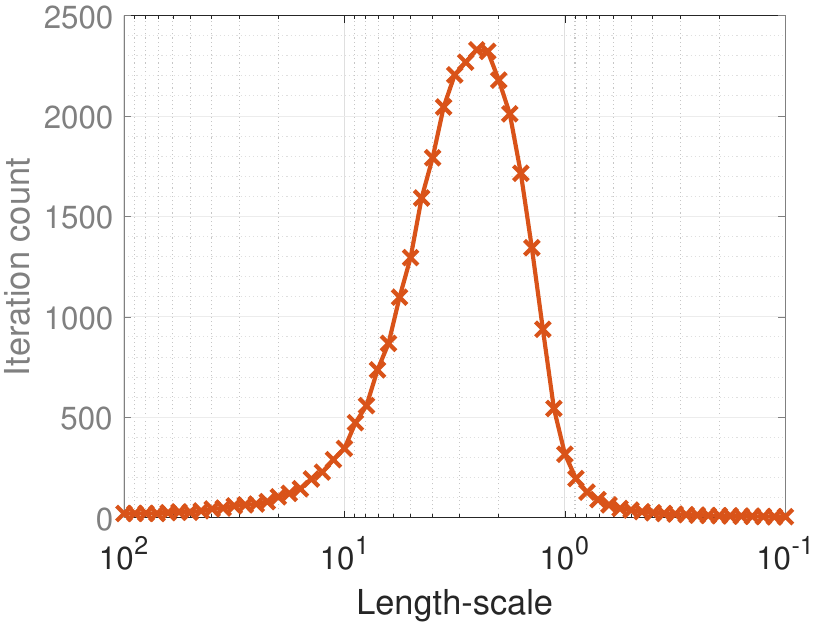}
\caption{Left: Spectrum of $61$ regularized Gaussian kernel matrices associated with the same $1000$ points sampled randomly over a cube with edge length $10$ and a fixed regularization parameter $\mu=0.0001$ but different length-scales $l$. Right: Iteration counts of unpreconditioned \texttt{CG} to solve Equation \eqref{eq:Problem} for the $61$ regularized kernel matrices to reach the relative residual tolerance $10^{-4}$. }
\label{fig:varyingl}
\end{figure}

In this paper, we seek a preconditioner for kernel matrix systems
\eqref{eq:Problem} that is adaptive to different kernel matrices $\mathbf{K}$
corresponding to different values of kernel function parameters such that it can flatten the curve in Figure \ref{fig:varyingl} (right).
When the numerical rank of $\mathbf{K}$ is small, there exist good methods
\cite{shabat_fast_2021,frangella_randomized_2021}
for preconditioning $\mathbf{K} + \mu \mathbf{I}$ based on a 
Nystr\"om approximation \cite{nys2001}
to the kernel matrix. We will provide a self-contained description of the Nystr\"om approximation in Section \ref{sec:nys}, as it is related to our proposed preconditioner. 

\section{Background: Nystr\"om approximation}
\label{sec:nys}
The Nystr\"om approximation, originally developed for approximating the eigenfunctions of integral operators, has emerged as an indispensable technique for generating low-rank approximations of kernel matrices. For a kernel matrix $\mathbf{K} \in \mathbb{R}^{n \times n}$, the Nystr\"om approximation takes the following form \cite{martinsson_tropp_2020}:
\begin{equation}
\label{eq:nys-def}
    \mathbf{K}_{\text{nys}} = \mathbf{K} \mathbf{X} (\mathbf{X}^{\top} \mathbf{K} \mathbf{X})^{\dagger} (\mathbf{K} \mathbf{X})^{\top},
\end{equation}
where $\mathbf{X} \in \mathbb{R}^{n \times k}$ is a test matrix, and $\cdot^{\dagger}$ denotes the pseudoinverse. The choice of $\mathbf{X}$ significantly influences the Nystr\"om approximation accuracy and efficiency. Employing $\mathbf{X}$ as a submatrix of a permutation matrix facilitates the sampling of $k$ columns from $\mathbf{K}$ and $\mathbf{X}^{\top} \mathbf{K} \mathbf{X}$ yields a $k \times k$ submatrix. This sampling-based Nystr\"om approximation offers a geometric insight when examining a kernel matrix over a dataset $X = \{\mathbf{x}_i\}_{i=1}^{n}$, equating the sampling of $k$ columns from $\bK$ to the selection of $k$ points from $X$. These points, referred to as landmark points, are crucial for the accuracy of the approximation.
Let $X_k$ represent the $k$ samples from the dataset $X$, and assume the $k\times k$ submatrix is invertible,  \eqref{eq:nys-def} can be rewritten as:
\begin{equation}
\label{eq:sampling-based-nystrom}
  \mathbf{K}_{\text{nys}} =  \mathbf{K}_{X,X_k} \mathbf{K}_{X_k,X_k}^{-1} \mathbf{K}_{X_k,X},
\end{equation}
where $\mathbf{K}_{X,Y}$ encapsulates $[\mathcal{K}(x,y)]_{x\in X,\,y\in Y}$ for any two datasets $X$ and $Y$. With the landmark points indexed first, $\mathbf{K}$ can be partitioned into a block 2-by-2  matrix:
\begin{equation}
  \mathbf{K} = 
 \begin{bmatrix}
  \mathbf{K}_{11} & \mathbf{K}_{12}\\
  \mathbf{K}_{12}^{\top} & \mathbf{K}_{22}
\end{bmatrix},
\end{equation}
where $\mathbf{K}_{11} = \mathbf{K}_{X_k,X_k}$, $\mathbf{K}_{12} = \mathbf{K}_{X_k,{X}\setminus X_{k}}$, and $\mathbf{K}_{22} = \mathbf{K}_{{X}\setminus X_{k},{X}\setminus X_{k}}$. Consequently, the sampling-based Nystr\"om approximation approximates $\mathbf{K}_{22}$ with $\mathbf{K}_{12}^{\top} \mathbf{K}_{11}^{-1} \mathbf{K}_{12}$:
\begin{equation}
\mathbf{K}_{\text{nys}} =  
\begin{bmatrix}
\mathbf{K}_{11} & \mathbf{K}_{12} \\
\mathbf{K}_{12}^{\top} & \mathbf{K}_{12}^{\top} \mathbf{K}_{11}^{-1} \mathbf{K}_{12}
\end{bmatrix}.
\end{equation}
This strategy underscores the importance of landmark point selection in formulating an effective approximation.
A different approach to constructing Nystr\"om approximations utilizes the projection method where alternative test matrices, such as Gaussian random matrices, are utilized in this context. However, the projection-based Nystr\"om approximation incurs a higher computational cost due to explicit matrix multiplications, unlike its sampling counterpart. Nevertheless, this method has stronger theoretical foundations, and the approximation error can be close to that of the best rank-$k$ approximation attainable via singular value decomposition (SVD) \cite{martinsson_tropp_2020}.
Note that direct implementation of projection-based Nystr\"om approximation following Equation \eqref{eq:nys-def} may result in numerical instability. As a result, the following form is recommended \cite{frangella_randomized_2021}:
\begin{equation}
\label{eq:random-nys}
\mathbf{K}_{\text{nys}} = \mathbf{U} \widehat{\mathbf{\Lambda}} \mathbf{U}^\top,
\end{equation}
where $\mathbf{U}$ comprises orthonormal columns, and $\widehat{\mathbf{\Lambda}}$ is a diagonal matrix. The  details for constructing $\mathbf{U}$ and $\widehat{\mathbf{\Lambda}}$ can be found in \cite{frangella_randomized_2021}. When employing Nystr\"om approximation as a preconditioner for Preconditioned CG (PCG) to solve \eqref{eq:Problem}, in conjunction with the Sherman–Morrison–Woodbury (SMW) formula and the orthonormality of $\mathbf{U}$, the inverse of the preconditioner $\mathbf{U} \widehat{\mathbf{\Lambda}} \mathbf{U}^\top + \mu \mathbf{I}$ can be expressed as:
\begin{equation}
(\mathbf{U} \widehat{\mathbf{\Lambda}} \mathbf{U}^\top + \mu \mathbf{I})^{-1} =
\mathbf{U}(\widehat{\mathbf{\Lambda}}+\mu\mathbf{I})^{-1}\mathbf{U}^\top+\frac{1}{\mu}(\mathbf{I}-\mathbf{U}\mathbf{U}^\top).
\label{eq:nystrom1-smw}
\end{equation}
Equation \eqref{eq:nystrom1-smw} is valid for any low-rank approximation of $\mathbf{K}$, provided $\mathbf{U}$ is orthonormal. While applying SMW formula directly to Equation \eqref{eq:sampling-based-nystrom} based on the sampling-based Nystr\"om preconditioner seems plausible,  this implementation is not numerically stable. Therefore,  \eqref{eq:nystrom1-smw} is preferred for both sampling-based and projection-based Nystr\"om preconditioners.
Preconditioners leveraging Nystr\"om approximations and other low-rank approaches for the kernel matrix $\mathbf{K}$ typically involve eigendecomposition or another factorization of a dense $k \times k$ matrix, which is feasible for small $k$ but becomes prohibitively expensive for large $k$ where $k$ is often set to be the numerical rank.  The numerical rank is defined as \(\#\{\lambda_i > c\mu \mid \lambda_i\) are the eigenvalues of \(\mathbf{K} + \mu\mathbf{I}, \forall i \in [n]\}\), where \(\#\{\cdot\}\) denotes the count of eigenvalues exceeding \(c\mu\). 
This definition bears resemblance to the concept of the effective dimension defined as $\Tr{\bK(\bK+c\mu)^{-1}} = \sum_{i=1}^{n}\frac{\lambda_i}{\lambda_i + c\mu}$ \cite{frangella_randomized_2021}. However, we found the numerical rank definition serves our purpose of adaptively estimating the rank better. In Section~\ref{sec:nystrom}, we propose a block 2-by-2 approximate factorization of $\mathbf{K} + \mu \mathbf{I}$ as a preconditioner, with the (1,1) block corresponding to the set of landmark points. Additionally, a method for estimating the number of landmark points is presented in Section~\ref{subsec: nys-rank}. The selection of landmark points based on Farthest Point Sampling (FPS) is supported by the analysis in Section~\ref{sec:theory}. The efficiency of the proposed preconditioner is demonstrated in Section~\ref{sec:experiment}, while Section \ref{sec:conclusion} summarizes this paper's contributions.


\section{Adaptive Factorized Nystr\"om (AFN) preconditioner}
\label{sec:nystrom}
This section outlines the development of the proposed preconditioner, initially motivated by the screening effect \cite{stein2002screening,stein20112010,stein2015does} and the FSAI method \cite{fsai}. Building on these concepts, we introduce the Adaptive Factorized Nystr\"om preconditioner (AFN), designed to enhance computational efficiency through a factorized matrix representation. To enhance performance across different kernel parameters, we also propose a rank estimation technique. 
\subsection{Screening effect and FSAI}
\label{subsec:screening effect and fsai}
In Gaussian processes, the kernel matrix \(\mathbf{K}_{XX}\) is used to characterize the covariance of the target values $\by = \{y_i\}_{i=1}^{n}$ of the data points $X = \{\bx_i\}_{i=1}^{n}$. Since $\Schur$ can be considered as the conditional covariance matrix of target values at $X\backslash X_k$ conditioned on those at $X_k$, 
the screening effect \cite{guinness2018permutation,pourahmadi1999joint,stein2002screening,stein20112010,stein2015does,schafer_compression_2021,schafer_sparse_2021} in geostatistics implies that optimal linear predictions at a point in a Gaussian process primarily rely on nearby data points. While the theory provides specific conditions for this effect, it is practically leveraged to improve the computational efficiency of Gaussian process regression. The Vecchia approximation \cite{vecchia_estimation_1988}, rooted in this concept, simplifies joint density calculations by conditioning on neighboring points, leading to a sparse Cholesky factorization of the precision matrix. However, the approximation accuracy depends on the strength of the screening effect and the number of neighboring points considered.  More specifically, the screening effect suggests that the optimal linear prediction of the target value $y_i$ at a point $\bx_i$ in a Gaussian process typically depends on the values at neighboring points $N_i = \{\bx_j|j\in D_{i}\}$ where $D_i$ is the indices of the nearest neighbors of $\bx_i$ which can be determined using K-Nearest Neighbors (KNN) algorithm. This has been theoretically scrutinized and conditions for its validity have been established, albeit under limited scenarios \cite{stein2002screening, stein20112010, stein2015does}. The Vecchia approximation \cite{vecchia_estimation_1988} utilizes this principle by approximating the exact joint density $p_1(\by)=p(y_1)\prod_{i=2}^{n}p(y_i | y_{1:i-1})\sim \mathcal{N}(\bo,\bK)$ with  $p_2(\by) = p(y_1) \prod_{i = 2}^{n} p(y_i | y_{N_i}) \sim \mathcal{N}(\bo , \hat{\bK})$, significantly simplifying calculations. This approximation yields a precision matrix  $\hat{\bK}^{-1}$   with a sparse Cholesky decomposition, where the Cholesky factor has a limited number of non-zero entries per row, equal to the size of $N_i$ \cite{datta2016hierarchical,katzfuss2021general}. While recent studies \cite{schafer_compression_2021,schafer_sparse_2021} confirm that the screening effect is valid for functions derived from Green's functions of elliptic operators, it is important to note that when the effect is weak or absent, the approximation will not be very accurate. Despite its theoretical limitations, the Vecchia approximation has recently achieved empirical success in numerous applications \cite{schafer_sparse_2021,katzfuss2021general,guinness2018permutation}. Remarkably, this approximation is equivalent to the Factorized Sparse Approximate Inverse (\texttt{FSAI}) method introduced by Kolotilina and Yeremin \cite{fsai}.
Therefore, based on the screening effect, we can use \texttt{FSAI} to compute
a sparse approximate inverse $\bG$ of the lower triangular Cholesky factor of
a symmetric positive definite (SPD) matrix $\bK_{22} + \noise -\bK_{21}(\bK_{11}+\noise)^{-1}\bK_{12}$, 
given a sparsity pattern $\mathbf{S}$ for $\bG$, i.e.,
$\mathbf{G^\top G} \approx (\bK_{22} + \noise -\bK_{21}(\bK_{11}+\noise)^{-1}\bK_{12})^{-1}$.
An important feature of \texttt{FSAI} is that the computation
of $\bG$ only requires the entries of $\bK_{22} + \noise -\bK_{21}(\bK_{11}+\noise)^{-1}\bK_{12}$ corresponding
to the sparsity pattern of $\bG$ and $\bG^\top$. This makes it possible to
economically compute $\bG$ even if $\bK_{22} + \noise -\bK_{21}(\bK_{11}+\noise)^{-1}\bK_{12}$ is large and dense.
Further, the computation of each row of $\bG$ is independent of other
rows and thus the rows of $\bG$ can be computed in parallel.  The nonzero pattern used for row $i$ of $\mathbf{G}$ corresponds to
the $w-1$ nearest neighbors of point $i$ that are numbered less than $i$
(since $\bG$ is lower triangular), where $w$ is a parameter which controls the nonzeros per row. The pseudocode of \texttt{FSAI} can be found in Algorithm \ref{alg:fsai}.
\begin{algorithm}[H]
\caption{Factorized Sparse Approximate Inverse (\texttt{FSAI})}
\begin{algorithmic}[1]
\STATE{\textbf{Input}: A symmetric positive definite matrix \(\mathbf{K}\), the number of nonzeros on each row $w$, a lower triangular sparsity pattern $\mathbf{S} \in \mathbb{R}^{n \times n}$ which is a zero-one matrix with at most $w$ ones on each row indicating the positions of $w$-nearest neighbors of each data point.
}
\FOR{$i=1$ to $n$}
    \STATE{Extract the non-zero pattern $\mathbf{s}_{i}$ from the $i$th row of $\mathbf{S}$ with length $w \ll n$}
    \STATE{Compute $\mathbf{G}_{i, \mathbf{s}_{i}} =\frac{\mathbf{e}_{w}^\top(\mathbf{K}_{\mathbf{s}_{i}, \mathbf{s}_{i}})^{-1}}{\sqrt{\mathbf{e}_{w}^{\top} (\mathbf{K}_{\mathbf{s}_{i},\mathbf{s}_{i}})^{-1}\mathbf{e}_{w}}}$}
\ENDFOR
\STATE{\textbf{Return}: $\mathbf{G}$}
\end{algorithmic}
\label{alg:fsai}
\end{algorithm}
In the next section, we will discuss how to apply \texttt{FSAI} on the $\Schur$ to improve the robustness and accuracy of the Nystr\"om approximation.
\subsection{AFN preconditioner construction and application}
{
We now propose a new preconditioner for $\bK + \mu \bI$ that can
be efficiently constructed even when $k$ is large. Recall that $\bK_{11}$ is the kernel matrix associated with 
a set of landmark points $X_k$. To control the computational cost, we impose a limit on the maximum size of $X_k$ setting it to a constant value, such as $2000$. For any SPD matrix $\mathbf{A}$, the Cholesky decomposition ensures that the matrix can be factored as $\mathbf{A} = \mathbf{L}\mathbf{L}^{\top}$, where $\mathbf{L}$ is a lower triangular matrix. Let $\mathbf{LL^\top}$ be the Cholesky factorization
of $\bK_{11} + \mu \bI$ and $\mathbf{G^\top G}$ be the \texttt{FSAI}
of $(\bK_{22}+\mu\mathbf{I}-\bK_{12}^{\top}(\bK_{11}+\mu\bI)^{-1}\bK_{12})^{-1}$.
Then we can define the following factorized preconditioner for $\bK + \mu \bI$:
\begin{equation}
\mathbf{M} =
\begin{bmatrix}
 \bL                   & \mathbf{0} \\
  \bK_{12}^{\top} \bL^{-\top} & \bG^{-1}
\end{bmatrix}
\begin{bmatrix}
  \bL^{\top} & \mathbf{L}^{-1}\bK_{12}\\
  \mathbf{0}         & \bG^{-\top}
\end{bmatrix} .
\label{eq:factorized-form}
\end{equation}
Expanding the factors,
\begin{align}
\mathbf{M}& =
\begin{bmatrix}
  \bK_{11} + \mu\bI & \bK_{12} \\
  \bK_{12}^{\top} & (\mathbf{G}^\top \mathbf{G})^{-1} + \bK_{12}^{\top} (\bK_{11}+\mu \bI)^{-1} \bK_{12} 
\end{bmatrix} \\
&= \Knys + \mu \bI + 
\underbrace{\begin{bmatrix}
   \bo & \bo \\
  \bo & (\mathbf{G}^\top \mathbf{G})^{-1} + \bK_{12}^{\top} \left((\bK_{11}+\mu \bI)^{-1}-(\bK_{11})^{-1}\right) \bK_{12}-\noise
\end{bmatrix}}_{\text{Correction term}},
\label{eq:realM}
\end{align}
we see that $\mathbf{M}$ equals 
$\Knys + \mu\bI$ plus a correction term. Thus the preconditioner is not a Nystr\"om preconditioner but has
similarities to it. 
Unlike a Nystr\"om preconditioner,
the factorized form approximates $\bK + \mu \bI$ entirely and does not
approximate $\bK$ separately, and thus avoids the SMW formula. In particular, when we have
$\mathbf{G^\top G} = (\bK_{22} + \noise -\bK_{12}^{\top}(\bK_{11}+\noise)^{-1}\bK_{12})^{-1}$ exactly, $\bM=\bK+\noise$. }

We call the preconditioner defined in \eqref{eq:factorized-form} the Adaptive Factorized Nystr\"om (\texttt{AFN}) preconditioner and summarize its construction procedure in Algorithm \ref{alg:afn}. Owing to the factorization structure, it is sufficient to store $\mathbf{L}$ and $\mathbf{G}$ from Algorithm \ref{alg:afn}. 
\begin{algorithm}
    \caption{Adaptive Factorized Nystr\"om (\texttt{AFN}) preconditioner construction}
    \begin{algorithmic}[1]
    \STATE{\textbf{Input}: Dataset $X = \{\bx_i\}_{i=1}^{n}$ for $\bx_i\in \mathbb{R}^{d}$, Kernel function $\mathcal{K}(\cdot,\cdot)$, regularization parameter $\mu$, estimated rank $k$ returned by Algorithm \ref{alg:nys-estimation}}.
    \IF{$d \leq 10$}
    \STATE Select $k$ landmark points denoted by $X_k$ according to FPS Algorithm \ref{alg:fps}.
    \ELSE
    \STATE Select $k$ landmark points denoted by $X_k$ using uniform sampling.
    \ENDIF
    \STATE{Perform Cholesky factorization: $\mathbf{L} = \operatorname{Chol}(\mathbf{K}_{11} + \mu \mathbf{I})$ where $\bK_{11} = \mathcal{K}(X_k,X_k)$ }
    \STATE{Invoke Algorithm \ref{alg:fsai} to compute $\mathbf{G} = \operatorname{FSAI}(\Schur)$ where $\bK_{12} = \mathcal{K}(X_k,X_r)$, $\bK_{22} = \mathcal{K}(X_r,X_r)$, $X_r = X\setminus X_k$.}
    \STATE{\textbf{Return}: Matrices $\mathbf{L}$ and $\mathbf{G}$}
    \end{algorithmic}
    \label{alg:afn}
    \end{algorithm}

The preconditioning operation for \texttt{AFN}
solves systems with the matrix $\mathbf{M}$. Assuming that the vectors $\br$
and $\bs$ are partitioned conformally with the block structure of $\mathbf{M}$,
then to solve the system
\[
\mathbf{M}
\begin{bmatrix}
\bs_1 \\ \bs_2
\end{bmatrix}
=
\begin{bmatrix}
\br_1 \\ \br_2
\end{bmatrix},
\]
{
the algorithm is
\begin{align*}
\bs_2 & := \bG^{\top}\bG \left( \br_2 - \bK_{12}^\top (\bL^{-\top} \bL^{-1}) \br_1 \right), \\
\bs_1 & := \bL^{-\top} \bL^{-1} ( \br_1 - \bK_{12} \bs_2 ) .
\end{align*}
}

The pseudocode detailing the application of the \texttt{AFN} preconditioner is described in Algorithm \ref{alg:pcg-afn-apply}.
\begin{algorithm}[H]
        \caption{Adaptive Factorized Nystr\"om (\texttt{AFN}) preconditioner application}
        \begin{algorithmic}[1]
        \STATE{\textbf{Input}: Vector $\mathbf{r}$, matrices $\mathbf{L}$, $\mathbf{G}$, $\mathbf{K}_{12}$}
        \STATE{Partition $\mathbf{r}$ conformally with the size of $\mathbf{L}$ and $\mathbf{G}$ as $[\mathbf{r}_1,\mathbf{r}_2]^{\top}$}
        \STATE{Solve $(\mathbf{K}_{11} + \mu \mathbf{I}) \mathbf{z} = \mathbf{r}_1$ by computing $\mathbf{z} = \mathbf{L}^{-\top}\mathbf{L}^{-1} \mathbf{r}_1$}
        \STATE{Compute  $\mathbf{s}_2 = \mathbf{G}^{\top}\mathbf{G}( \mathbf{r}_2 - \mathbf{K}_{12}^\top \mathbf{z} )$}
        \STATE{Solve $(\mathbf{K}_{11} + \mu \mathbf{I}) \mathbf{s}_1 = (\mathbf{r}_1 - \mathbf{K}_{12} \mathbf{s}_2 )$ by computing $\mathbf{s}_1 = \mathbf{L}^{-\top} \mathbf{L}^{-1} (\mathbf{r}_1 - \mathbf{K}_{12} \mathbf{s}_2 )$}
        \STATE{\textbf{Return}: Vector $\mathbf{s} = [\mathbf{s}_1, \mathbf{s}_2]^{\top}$}
        \end{algorithmic}
        \label{alg:pcg-afn-apply}
        \end{algorithm}
In the next two sections, we will discuss two pivotal components in the construction of \texttt{AFN}: the selection of an optimal number of landmark points and their identification method. Addressing these components is essential for refining the Nystr\"om approximation and, by extension, the efficacy of the \texttt{AFN} method. The choice of landmark points, specifically the determination of the block size of \(\mathbf{K}_{11}\), plays a critical role. An underestimation may lead to inadequate approximations, whereas an overestimation could introduce unnecessary computational burdens and potential numerical instability. To navigate this challenge, we propose Algorithm~\ref{alg:nys-estimation} for empirically estimating the optimal number of landmark points, a process elaborated in Section~\ref{subsec: nys-rank}.
It is also equally crucial to devise a strategy for selecting these landmark points. For low dimensional data, we advocate for FPS due to its balance of simplicity, efficiency, and effectiveness, a rationale further explored in Section \ref{sec:theory}. Conversely, for datasets with high-dimensional features, we recommend uniform sampling as a means to alleviate the computational cost.
\subsection{Adaptive choice of approximation rank}
\label{sec:adaptive}
To construct a preconditioner that is adaptive and efficient
for a range of regularized kernel matrices arising from different values of the
kernel function parameters, it is necessary to estimate the rank
of the kernel matrix $\bK$. {For example, if the estimated rank is small enough that it is inexpensive to perform an eigendecomposition of a $k$-by-$k$ matrix, then the Nystr\"om preconditioner should be used due to the reduced construction cost.}

\label{subsec: nys-rank}
It is of course too costly in general to use a rank-revealing
decomposition of $\bK$ to compute $k$. Instead, we will compute $k$
that approximately achieves a certain Nystr\"om approximation accuracy via checking the relative Nystr\"om approximation error on a subsampled dataset.

First, a dataset $X_m$ of $m$ points is randomly subsampled from $X$. The number of points $m$ is an input to the procedure, and $m$ can be much smaller than the $k$ that will be computed. Then the coordinates of the data points in $X_m$ are scaled by $(m/n)^{1/d}$ and the smaller kernel matrix $\bK_{X_m,X_m}$ is formed. The rationale of this scaling is that we expect the spectrum of $\bK_{X_m,X_m}$ has a similar decay pattern as that of $\bK_{X,X}$. We now run FPS on $X_m$ to construct Nystr\"om approximations with increasing rank to $\bK$ until the relative Nystr\"om approximation error falls below $0.1$ and define this Nystr\"om rank as $r$. Finally, we approximate the Nystr\"om rank of $\bK$ as $rn/m$. Figure \ref{fig:fd2nyserror} plots the Nystr\"om approximation errors on subsampled matrices and original matrices associated with two different length-scales. The data points $X$ are generated randomly by sampling $1000$ points uniformly within a cube and $m=100$ points are subsampled randomly. The two relative Nystr\"om approximation error curves show a close match in both cases. This rank estimation method is summarized in Algorithm \ref{alg:nys-estimation}. We also find that if the estimated rank is small (e.g., less than $2000$), we can perform an eigen-decomposition of $\bK_{{X}_m,{X}_m}$ associated with the unscaled data points and refine the estimation with the number of eigenvalues greater than $0.1\mu$.  Here we estimate the numerical rank defined in the previous section with $c = 0.1$ to have a refined rank estimation to further enhance performance.
\begin{figure}[t]
     \begin{minipage}{0.47\textwidth}
      \centering
      \includegraphics[width=0.99\linewidth]{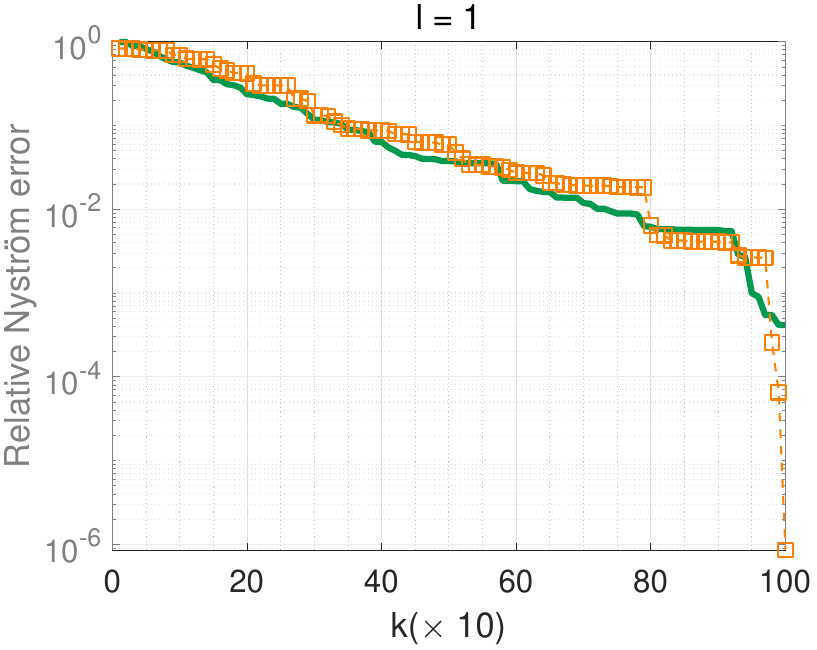}
   \end{minipage}
     \begin{minipage}{0.47\textwidth}
      \centering
      \includegraphics[width=.99\linewidth]{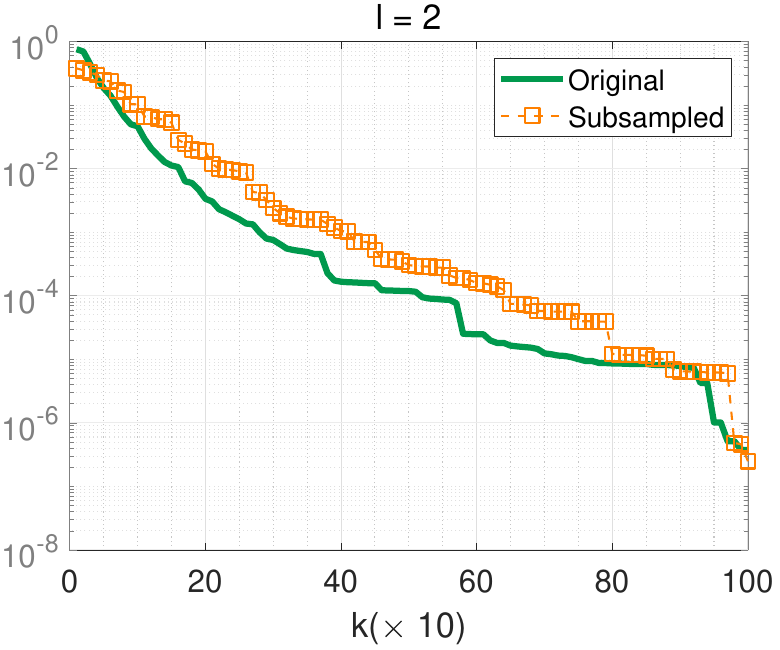}
   \end{minipage}
\caption{Comparison of the relative Nystr\"om approximation error curves for an original dataset and a subsampled dataset with $100$ points, associated with two different length-scales. The original dataset contains 1000 uniformly sampled points from a cube with edge length $10$. The indices of the subsampled dataset are matched with those of the original dataset by computing the  relative Nystr\"om approximation errors on the original dataset only for ranks that are multiples of 10. The plot shows how the approximation error changes as the rank of the approximation increases.}

\label{fig:fd2nyserror}
\end{figure}

\begin{algorithm}
\caption{Nystr\"om rank estimation}
\begin{algorithmic}[1]
\STATE \textbf{Input}: Dataset $X$ with size $n$, subsample size $m$, and kernel function $\mathcal{K}(\bx,\by)$
\STATE \textbf{Output}: Approximate Nystr\"om rank $k$
\STATE Randomly subsample a subset $X_m$ of $m$ points from $X$ and scale the coordinates of $X_m$ by $({m}/{n})^{1/d}$
\STATE Form the $m\times m$ matrix $\bK_{{X}_m,{X}_m}$
\STATE Find the Nystr\"om rank $r$ such that the relative Nystr\"om approximation error for $\bK_{{X}_m,{X}_m}$ with FPS sampling implemented in Algorithm \ref{alg:fps} falls below $0.1$
\IF {$k \ge 2000$}
  \STATE \textbf{Return} $k={rn}/{m}$ 
\ELSE
  \STATE Compute eigenvalues of $\bK_{{X}_m,{X}_m}$ associated with the unscaled data points 
  \STATE Return $k$, the number of eigenvalues greater than 0.1$\mu$
\ENDIF
\end{algorithmic}
\label{alg:nys-estimation}
\end{algorithm}

If the estimated rank $k$ is smaller than $2000$, then the Nystr\"om preconditioner should be used. \texttt{AFN} is only constructed when the estimated rank exceeds $2000$ for better efficiency. The selection of the preconditioning method is shown precisely in Algorithm \ref{alg:strategy}.

The choice of the landmark points affects the accuracy of the
overall \texttt{AFN} preconditioner, just as this choice affects
the accuracy of the Nystr\"om preconditioners.  The sparsity and the conditioning of $\Schur$ generally
improves when more landmark points are chosen, which 
would on the other hand increase the computational cost and the instability of the Cholesky factorization of $\bK_{11}+\mu\bI$. Finally, we combine every component and summarize the PCG with \texttt{AFN} in Algorithm \ref{alg:strategy}. In the next section, the choice of landmark points
is discussed in light of these considerations. 
\begin{algorithm}
\caption{Preconditioned conjugate gradient with the proposed preconditioning scheme}
\begin{algorithmic}[1]
        \STATE{\textbf{Input}: Kernel matrix $\mathbf{K}$, regularization parameter $\mu$, right-hand side vector $\mathbf{b}$} \STATE Estimate numerical rank $k$ of $\bK$ with Algorithm \ref{alg:nys-estimation}
\IF {$k \ge 2000$}
 \STATE{Solve $(\mathbf{K} + \mu\mathbf{I}) \mathbf{a} = \mathbf{b}$ using PCG with the \texttt{AFN} preconditioner, applied as per Algorithm \ref{alg:pcg-afn-apply}}
        \ELSE
            \STATE{Solve $(\mathbf{K} + \mu\mathbf{I}) \mathbf{a} = \mathbf{b}$ using PCG with the column sampling-based Nystr\"om preconditioner, applied as per Equation \eqref{eq:nystrom1-smw}}
\ENDIF
\STATE \textbf{Return}: approximate solution vector
\end{algorithmic}
\label{alg:strategy}
\end{algorithm}

\section{Selecting the landmark points}
\label{sec:theory}
\shifanrewriting{
Existing methodologies for sampling $k$ landmark points from a dataset with $n$ data points include uniform sampling \cite{nys2001}, the anchor net method \cite{DDFac,anchornet}, leverage score sampling \cite{nys2005,leverage2016,nys2017}, $k$-means-based sampling \cite{nys2010kmeans}, determinantal point process (DPP)-based sampling \cite{belabbas_spectral_2009}, and random pivoted Cholesky sampling \cite{chen2022randomly}.
Uniform sampling with the computational complexity $O(k)$ excels in scenarios such as kernel ridge regression applications where direct access to kernel matrices is available, and the data does not exhibit unbalanced clusters.  Nonetheless, its efficacy diminishes when faced with unbalanced clusters as it tends to oversample larger clusters.
To address this shortcoming, adaptive sampling techniques have been proposed. These methods, including leverage score sampling, DPP-based sampling, and random pivoted Cholesky sampling, employ non-uniform sampling distributions derived from kernel matrices. For instance, ridge leverage score sampling constructs the probability for sampling the $i$-th column proportional to the $i$-th diagonal entry of $(\bK + \mu \bI)^{-1} \bK$. In \cite{musco2017recursive}, a recursive sampling strategy was introduced, reducing the computational cost of ridge leverage score sampling to $O(nk)$ kernel evaluations and $O(nk^2)$ running time. $k$-DPP-based sampling extends the sampling distribution across all $k$-subsets of ${1, \dots, n}$, albeit at a much higher computational cost of $O(n^3)$. However, a Markov Chain Monte Carlo (MCMC) approach proposed in \cite{li_fast_2016} can reduce this cost to linear time under some conditions. Due to the challenges of verifying these conditions and the necessity to reevaluate a $k\times k$ determinant, $k$-DPP-based sampling has experienced limited acceptance in practice compared to other sampling methods.
Random pivoted Cholesky sampling, as presented in \cite{chen2022randomly}, introduces a method aligned with pivoted Cholesky procedures, where the $i$-th pivot is selected proportional to the magnitude of the diagonal entries of the Schur complement at  the $i$-th step. This method necessitates $O(n(k+1))$ kernel evaluations. 
Geometry-based sampling is another avenue, with $k$-means sampling clustering data points into $k$ clusters and utilizing the centroids as landmark points at a cost of $O(tkn)$, where $t$ represents the iteration count in Lloyd's algorithm. The Anchor Net method \cite{DDFac}, an efficient tactic to mitigate the limitations of uniform sampling in high-dimensional datasets, employs a low-discrepancy sequence to  diminish gaps and clusters compared to uniform sampling while maintaining robust space coverage, at a complexity of $O(nk)$.
In our proposed preconditioner, a few different sampling methods can be employed. We opt for Farthest Point Sampling (FPS) due to its simplicity, ease of use, cost-effectiveness, and independence from the length-scale parameter. Specifically, landmark points will be selected based on a balance between two geometric measures to ensure the preconditioner's effectiveness and robustness.
}

The first measure $h_{X_k}$, called \emph{fill
distance} \cite{meshfree,lazzaro2002radial}, is used to quantify how
well the points in $X_k$ fill out a domain $\Omega$:
\begin{equation}
h_{X_k} = \max_{\bx\in \Omega \backslash X_k} \operatorname{dist}(\bx, X_k),
\label{eq:filldistance}
\end{equation}
where
$\operatorname{dist}(\mathbf{x},Y) = \inf_{\by \in Y} \Vert \bx-\by\Vert$ is the distance between a point $\mathbf{x}$ and a set $Y$,
and where $\Omega$ denotes the domain of the kernel function under consideration which can be either a continuous region or a finite discrete set. {The geometric interpretation of this measure is the largest radius of an empty ball in $\Omega$ that does not intersect with $X_k$. This implies $X_k$ with a smaller fill distance will better fill out $\Omega$.


The second measure $q_{X_k}$, called \emph{separation distance} \cite{meshfree,lazzaro2002radial}, is defined as the distance between the closest pair of points in $X_k$:
\begin{align}
\label{eq-def: separation distance}
    q_{X_k} =\min_{\bx_{k_i},\bx_{k_j} \in X_{k}, k_i \neq k_j} \operatorname{dist}(\bx_{k_i},\bx_{k_j}).
\end{align}
The geometric interpretation of this measure is the diameter of the largest ball that can be placed around every point in $X_k$ such that no two balls overlap. A larger $q_{X_k}$ indicates that the columns in $\bK_{11}$ tend to be more linearly independent and thus leads to a more well-conditioned $\bK_{11}$. {Given that the separation distance serves as a metric for the conditioning of the kernel matrix \cite{wendland2006computational} and the conditioning of $\bK_{11}$ will affect the numerical stability of $\mathbf{L}$, a larger separation distance implies a more stable Nystr\"om approximation and a more stable AFN preconditioner.}

As more landmark points are sampled, both $h_{X_k}$ and $q_{X_k}$ tend to decrease. We wish to choose $X_k$ such that $h_{X_k}$ is small while $q_{X_k}$ is large. We will analyze the interplay between $h_{X_k}$ and $q_{X_k}$ in Section \ref{subsec:nearopt}. In particular, we will show that if $h_{X_k}\leq C q_{X_k}$ for some constant $C$, then $h_{X_k}$ and $q_{X_k}$ have the same order as the minimal value of the fill distance and the maximal value of the separation distance that can be achieved with  $k$ points, respectively.

Moreover, we find that FPS \cite{FPS97} can generate landmark points with $h_{X_k}\leq q_{X_k}$. FPS is often used in mesh generation \cite{FPS06} and computer graphics \cite{FPS11}. {In spatial statistics, FPS is also known as MaxMin Ordering (MMD) \cite{guinness2018permutation}.} FPS initializes $X_1$ with an arbitrary point $\bx_{k_1}$ in $X$ (better choices are possible). At step $i+1$, FPS selects the point that is farthest away from $X_{i}$
\begin{equation}
    \bx_{k_{i+1}} = \argmax_{\bx\in X \backslash X_{i}} \operatorname{dist}(\bx,X_{i}).
\end{equation}
See Figure \ref{fig:point-selection} for an illustration of FPS on a two-dimensional dataset and {the complete pseudocode of FPS in Algorithm \ref{alg:fps}.}
\begin{figure}
     \centering
     \includegraphics[width=.3\linewidth]{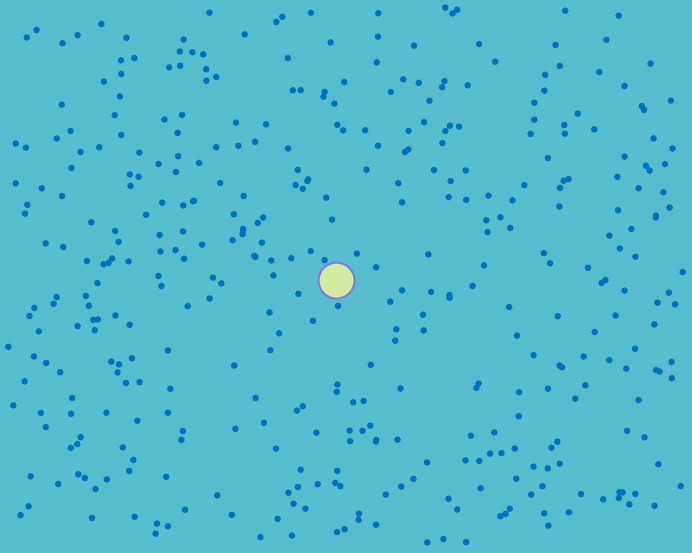} \hspace{0.5em}
     \includegraphics[width=.3\linewidth]{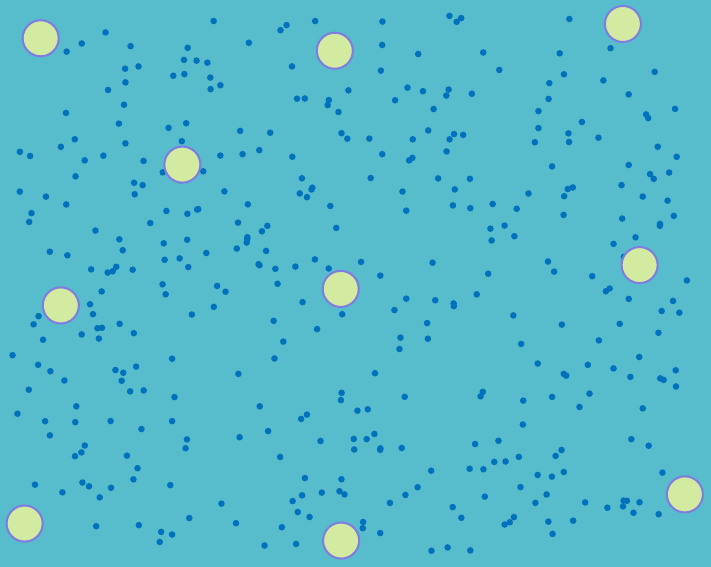} \\[1em]
     \includegraphics[width=.3\linewidth]{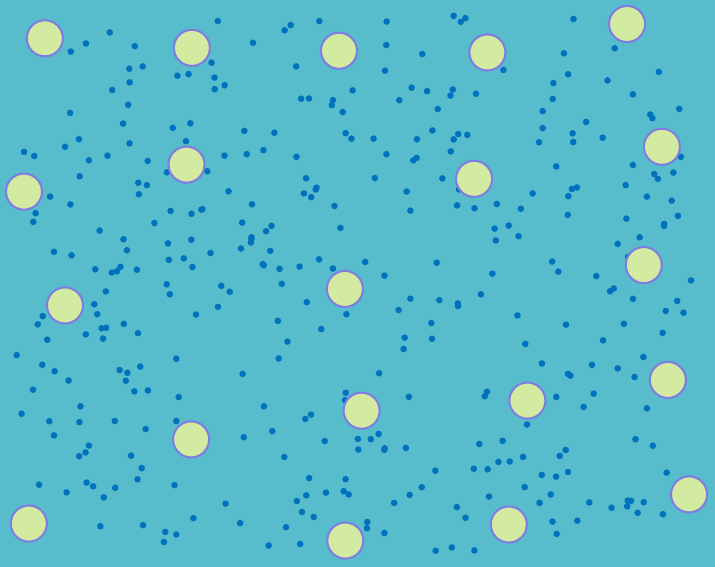} \hspace{0.5em}
     \includegraphics[width=.3\linewidth]{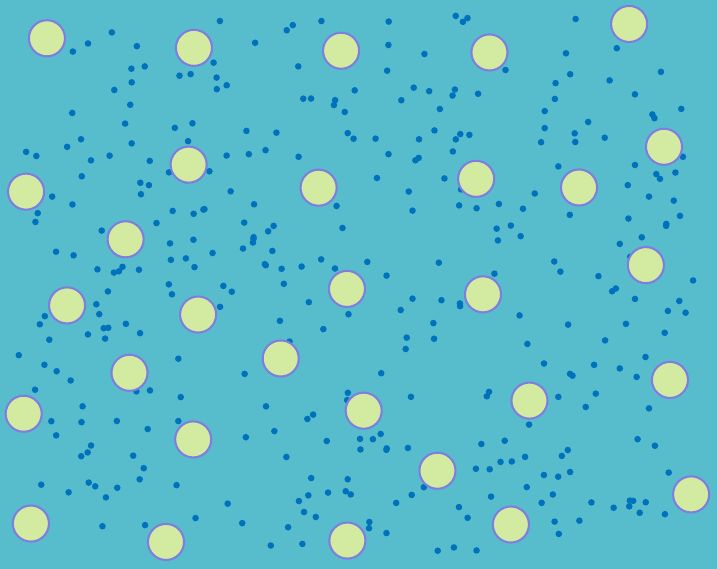}
    \caption{An illustration of FPS for selecting one, ten, twenty and thirty points from a two-dimensional dataset with $400$ points where the big circles represent the selected points and the dots denote the other data points.}
    \label{fig:point-selection}
\end{figure}
The landmark points, selected through  FPS, are distributed evenly across the dataset, avoiding the formation of dense clusters. This property will be analyzed in the next few sections.

\subsection{Interplay between fill and separation distance}
\label{subsec:nearopt}
{In this section, we will study the relationship between $h_{X_{k}}$ and $q_{X_k}$. We will show that if $h_{X_k}\leq Cq_{X_k}$ for a constant $C$, then $h_{X_k}$ and $q_{X_k}$ will have the same order as the minimal fill distance and maximal separation distance that can be achieved with any subset with $k$ points, respectively.}

First notice that there exist a lower bound for $h_{X_{k}}$ and a upper bound for $q_{X_k}$, which is analyzed in the next theorem when all the points are inside a unit ball in $\mathbb{R}^d$.
\begin{theorem}
\label{thm:fill-separation worst case}
Suppose all the data points are inside a unit ball $\Omega$ in $\mathbb{R}^{d}$. Then for an arbitrary subset $X_{k} = \{\bx_{i_1},\dots,\bx_{i_k}\}$ of $X$, the following bounds hold for $h_{X_{k}}$ and $q_{X_{k}}$:
{
\begin{align}
    h_{X_{k}} \geq   k^{-1/d}\quad  \text{and} & \quad q_{X_{k}}\leq 6 k^{-1/d}.
\end{align}
}
\end{theorem}
\begin{proof}
In order to show the lower bound of $h_{X_k}$, we first derive an upper bound of the volume of $\Omega$. Notice that $\Omega\subset \bigcup_{i = 1}^{k}B_{h_{X_k}}(\bx_{k_i})$ where $B_{h_{X_{k}}}(\bx_{k_i})$ is the ball centered at $\bx_{k_i}$ with radius $h_{X_k}$. Then 
\begin{align*}
  \text{Vol}(\Omega) &\leq \sum_{i = 1}^{k}\text{Vol}(B_{h_{X_k}}(\bx_{k_i}))= k\frac{\pi^{d/2}}{\Gamma(\frac{d}{2}+1)}h_{X_k}^{d}.
\end{align*}
This gives us the first bound.

Similarly, we get an upper bound of $q_{X_k}$ by using the packing number bound from \cite{wu2016packing}:
\begin{align*}
    k  \leq \Big(\frac{3}{\frac{q_{X_k}}{2}}\Big)^{d}
\end{align*}
Thus we have
\begin{align*}
     q_{X_k} \leq {6}k^{-\frac{1}{d}}.
\end{align*}

This gives us the second bound.
\end{proof}
\begin{remark}
When $\Omega$ satisfies the interior cone condition \cite{muller_komplexitat_2009},
similar bounds $ h_{X_{k}} \geq  C_{\Omega}k^{-1/d} ~ \text{and}  ~q_{X_{k}}\leq C'_{\Omega} k^{-1/d}$ can be derived for more complex bounded domains where $C_{\Omega}~ \text{and}  ~ C'_{\Omega}$ are two constants depending on the domain $\Omega$.
\end{remark}

{The above bounds show that the minimal fill distance $h_{X_k}$ cannot be smaller than $k^{-1/d}$ while the maximal separation distance $q_{X_k}$ cannot be greater than ${6}k^{-1/d}$ and $\frac{1}{6}q_{X_k}\leq h_{X_{k}}$ when the domain is a unit ball in $\mathbb{R}^d$.} In the following theorem, we show that if a sampling scheme can select a subset $X_k$ with $h_{X_k}\leq Cq_{X_k}$, then $q_{X_k}$ has the same order as the maximal separation distance that can be achieved by a subset with $k$ points.

\begin{theorem}
\label{thm:FPS is near optimal}
Assume the data points are on a bounded domain $\Omega$ that satisfies the interior cone condition, then if $h_{X_k}\leq Cq_{X_k}$
\begin{align}
    C_{\Omega}  k^{-1/d} \leq h_{X_k} \leq C \times C_{\Omega}^{\prime} k^{-1/d}   ,  & \quad \frac{C_{\Omega}}{C} k^{-1/d}\leq q_{X_k}\leq C_{\Omega}^{\prime} k^{-1/d}.
\end{align}
\end{theorem}
\begin{proof}
If $ h_{X} \leq C q_{X}$, then we have
\begin{align*}
    C_{\Omega}k^{-1/d} \leq  h_{X_{k}} \leq C q_{X_{k}}\leq C\times C_{\Omega}^{\prime} k^{-1/d}.
\end{align*}
\end{proof} 
Theorem \ref{thm:FPS is near optimal} shows that $h_{X_k}$ is at most $C\times\frac{C^{'}_{\Omega}}{C_{\Omega}}$ times larger than its theoretical lower bound and $q_{X_k}$ is at least $\frac{1}{C}\times\frac{C_{\Omega}}{C_{\Omega^{'}}}$ times as large as its theoretical upper bound in this case.

\subsection{Optimal properties of FPS}
\label{subsec:fps}
Although FPS is a greedy algorithm designed to select a set of data points
with maximal dispersion at each iteration, FPS can generate $X_k$
with $h_{X_k}$ at most $2$ times the minimal fill distance
\cite{gonzalez_clustering_1985} and $q_{X_k}$ at least half the
largest separation distance over all subsets with $k$ points
\cite{white1991maximal}. 
\shifanrewriting{In the next theorem, we first confirm that FPS can generate $X_k$ satisfying $h_{X_k} \leq q_{X_k}$ and then demonstrate its two near-optimality properties in a cohesive manner. While these properties have been independently established in \cite{gonzalez_clustering_1985, white1991maximal}, our analysis amalgamates and revalidates these results within a unified framework.}


\begin{theorem}
\label{thm: FPS 2 times optimal}
Suppose the minimal fill distance of a subset with $k$ points is achieved with $X^{\ast}_k$ and the maximal separation distance of a subset with $k$ points is achieved with $X_{k\ast}$. Then the set $X_k$ sampled by FPS satisfies 
\begin{align}
   h_{X_k}\leq q_{X_k}\quad \text{and} \quad q_{X_k}\geq\frac{1}{2}q_{X_{k*}} \quad \text{and} \quad h_{X_k} \leq 2h_{X_k^{\ast}}.
\end{align}
\end{theorem}
\begin{proof}
Without loss of generality, we assume the subset $X_k$ sampled by FPS contains the points $\bx_1,\bx_2,\ldots,\bx_k$. Suppose $q_{X_{k}} = \operatorname{dist}(\bx_j,\bx_m)$ with $j< m<(k+1)$, and point $\bx_m$ is selected at iteration $m$ by FPS, then 
\begin{equation}
\label{eq:qx is hx}
  h_{X_{m-1}} = \max_{\bx\in X \backslash X_{m-1}} \operatorname{dist}(\bx, X_{m-1})=\operatorname{dist}(\bx_j,\bx_m)=q_{X_k}.  
\end{equation}
Since $h_{X_k}$ is a non-increasing function of $k$, we have $h_{X_{k}}\leq h_{X_{m-1}}=q_{X_k}$. 

We now prove $q_{X_{k}} \geq \frac{1}{2} q_{X_k*}$. According to the definition, there exists a subset with $k$ points $X_{k*} = \{\bx^{1}_{*},\dots, \bx^{k}_{*}\}$ such that 
$$
q_{X_{k*}} = \max_{Y\subset X, |Y|=k } \; \min_{\bx_i,\bx_j\in Y}\operatorname{dist}(\bx_i, \bx_j).
$$
According to \eqref{eq:qx is hx}, we know all the points in $X$ must lie in one of the $m-1$ disks defined by 
\begin{equation}
    C(\bx_i, q_{X_k}) = \{\bx| \norm{\bx-\bx_i} \leq q_{X_k}\},\quad i \in [m-1].
\end{equation}
Since $m-1<k$, at least two points $\bx^{i}_{*}, \bx^{j}_{*}\in X_{k*}$ must belong to the same disk centered at some $\bx_l$. Therefore, $2q_{X_k} \geq \operatorname{dist}(\bx^{i}_{*}, \bx_l) + \operatorname{dist}(\bx^{j}_{*}, \bx_l) \geq \operatorname{dist}(\bx^{i}_{*}, \bx^{j}_{*}) \geq q_{X_k*} $ via the triangle inequality.

Next, we prove $h_{X_k} \leq 2h_{X_k^{\ast}}$. At the $k$th iteration of FPS, the set $X$ can be split into $k$ clusters $\{C_i\}_{i=1}^{k}$ such that the point $\mathbf{x}$ in $X$ will be classified into cluster $C_i$ if $\operatorname{dist}(\bx_i,\mathbf{x})\leq \operatorname{dist}(\bx_j,\mathbf{x}),~\forall j\neq i$. At the $(k+1)$th iteration of FPS, one more point $\bx_{k+1}$ will be selected. Then we can show that $$\operatorname{dist}(\bx_i,\bx_j)\geq h_{X_k} \quad \text{for} \quad i,j\in \{1,2,\ldots,k+1\},$$
and in particular
$$q_{X_{k+1}}\geq h_{X_k}.$$ 
Assume $\bx_{k+1}\in C_i$. From the definition of $h_{X_k}$, we know that $\operatorname{dist}(\bx_{k+1},\bx_i)= h_{X_k}$ and $\operatorname{dist}(\bx_{k+1},\bx_j)\geq \operatorname{dist}(\bx_{k+1},\bx_i)$ for $j\neq i$. Moreover, we have $\operatorname{dist}(\bx_i,\bx_j)\geq q_{X_k}$ for $j\neq k+1$. Since $q_{X_k}\geq h_{X_k}$, we know  $q_{X_{k+1}}=\operatorname{dist}(\bx_i,\bx_j)\geq h_{X_k}$. 

Finally, assume $\bx^{\ast}_1,\bx^{\ast}_2,\ldots,\bx^{\ast}_k$ are the optimal subset of $X$ that achieves the minimal fill distance with cardinality $k$. Now the set $X$ can be split into $k$ clusters $\{{C}^{\ast}_i\}_{i=1}^{k}$ such that the point $\mathbf{x}$ in $X$ will be classified into $C^{\ast}_i$ if $\operatorname{dist}(\bx_i^{\ast},\mathbf{x})\leq \operatorname{dist}(\bx_j^{\ast},\mathbf{x}),~\forall j\neq i$. Assume the points selected by FPS in the first $k+1$ iterations are $\bx_1,\bx_2,\ldots,\bx_{k+1}$. We know that at least two points from $\bx_1,\bx_2,\ldots,\bx_{k+1}$ belong to the same cluster. Denote these two points as $\bx_p$ and $\bx_q$ and the corresponding cluster is $C^{\ast}_j$. Then we have 
\[
h_{X_{k}}\leq q_{X_{k+1}}\leq \operatorname{dist}(\bx_p,\bx_q) \leq \operatorname{dist}(\bx_p,\bx^{\ast}_i) + \operatorname{dist}(\bx_q,\bx^{\ast}_i) \leq 2h_{X^{\ast}_k},
\]
which indicates that $h_{X_{k}} \leq 2h_{X^{\ast}_k}$.
\end{proof}

\subsection{Nystr\"om approximation error analysis based on fill distance}
\label{subsec:apprerr}
{ Then we justify why FPS is a good landmark points selection method from the perspective of Nystr\"om approximation error.
Define the Nystr\"om approximation error as 
$$\norm{\bK-\Knys } = \norm{\bK_{22} - \bK_{21}\bK_{11}^{-1}\bK_{12}}.$$ In this section, we will show that the Nystr\"om approximation error is also related to the fill distance $h_{X_k}$.  In particular, for Gaussian kernels defined in \eqref{eq:gaussianf} and inverse multiquadric kernels \begin{equation}
    \mathcal{K}(\bx, \by) = (c^2 + \norm{\bx-\by}^2)^{-\frac{p}{2}},~p>0,~c\in \mathbb{R},
    \label{eq:inverseMultiquadric}
\end{equation} we can derive a Nystr\"om approximation error estimate in terms of the fill distance, as presented in the following theorem.}
{
\begin{theorem}
\label{thm:nystrom-bound}
The Nystr\"om approximation $\Knys = \bK_{{X},X_k} \bK_{X_k,X_k}^{-1}\bK_{X_k,X}$ to $\bK$ using the landmark points $X_{k} = \{\bx_{k_i}\}_{i=1}^{k}$ has the following error estimate
\label{corollary:approximation fill distance 2-norm bound}
\begin{equation}
\label{ineq: discrete fill distance bound}
   \Vert \bK- \Knys\Vert < \sqrt{n\Vert \bK \Vert}  C'\exp(-C{''}/h_{X_k}),
\end{equation}
where $C'$ and $C{''}$ are constants independent of $X_k$ but dependent on kernels and the domain.
\end{theorem}}
The detailed proof of Theorem \ref{thm:nystrom-bound} is in Appendix \ref{sec:fast_fps}. {This theorem is a discrete version of the Theorem A in \cite{belkin_approximation_2018}, which implies that kernel operators corresponding to smooth kernels are effective low rank.  Our proof adapts the original results on kernel functions, as presented in \cite{belkin_approximation_2018}, to discrete matrix settings. This extension shows that the low-rank approximation mentioned in \cite{belkin_approximation_2018} can indeed be interpreted as a Nystr\"om approximation applicable to matrices.} For this Nystr\"om approximation. Theorem \ref{thm:nystrom-bound} implies landmark points $X_k$ with a smaller fill distance can yield a more accurate Nystr\"om approximation. {We illustrate this numerically with an experiment. In Figure \ref{fig:fill-distance-and-error temp},  we plot the fill distance curve and the Nystr\"om approximation error curve corresponding to a Gaussian kernel with $l=10$ when $1000$ points are uniformly sampled from a cube with edge length $10$. We test random sampling and FPS for selecting the landmark points and observe that FPS leads to a smaller fill distance than random sampling. We also observe that FPS Nystr\"om can achieve lower approximation errors than the randomly sampled one when the same $k$ is used.}
Thus we will use FPS to select landmark points in the construction of Nystr\"om-type preconditioners if the estimated rank is small. Meanwhile, the rank estimation algorithm discussed in Section \ref{sec:adaptive} also relies on FPS.

\begin{figure}[t]
     \includegraphics[width=.47\textwidth]{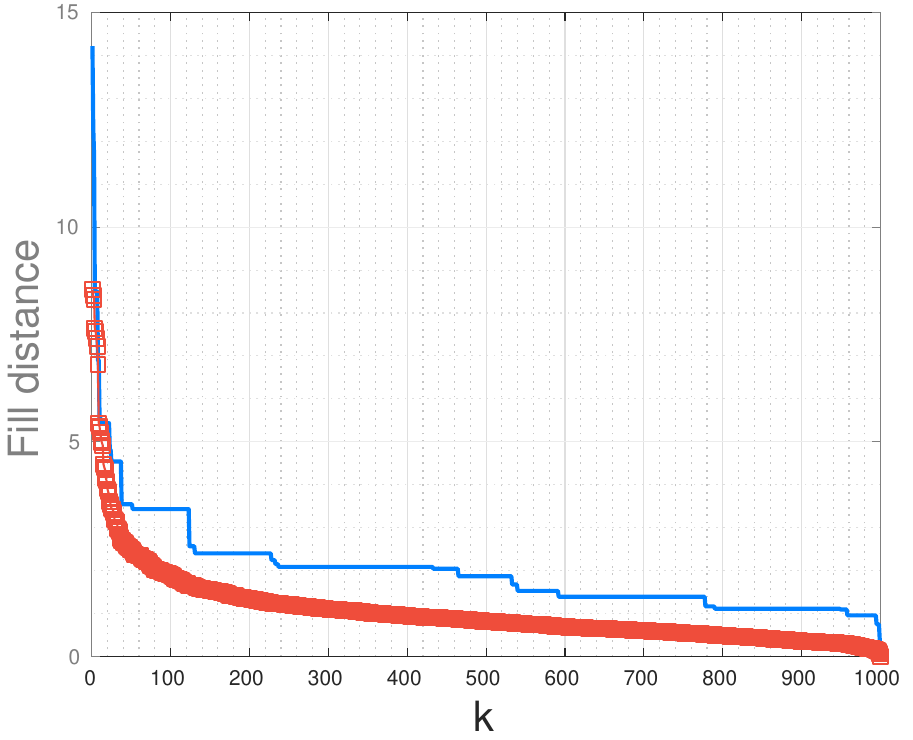}
     \includegraphics[width=.47\textwidth]{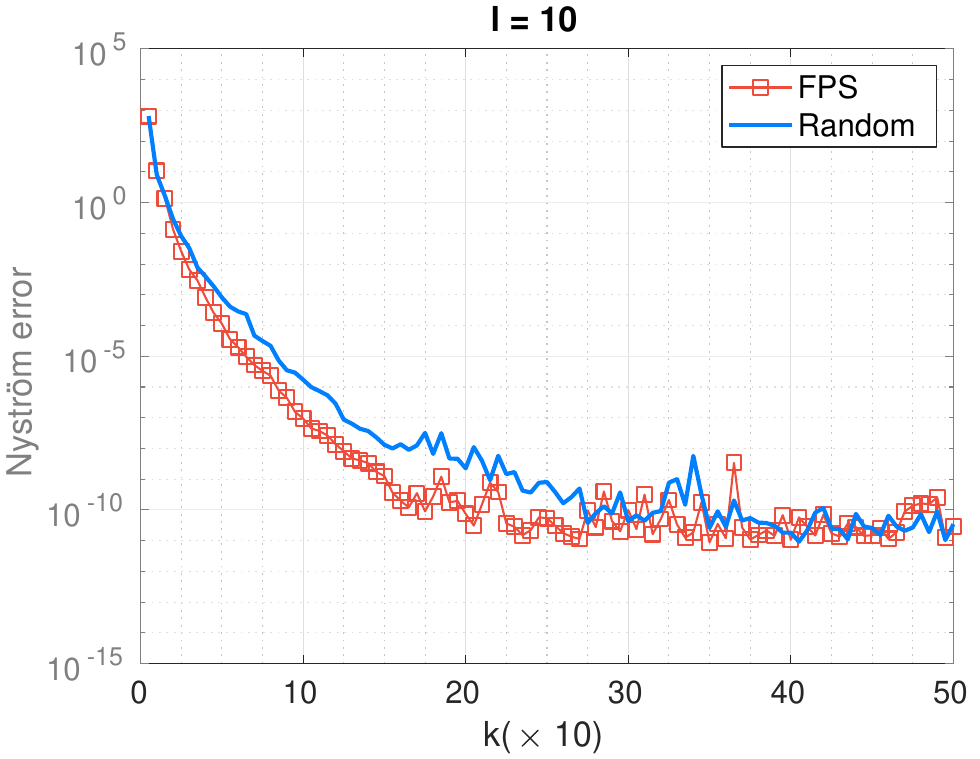}
    \caption{Comparison of fill distance and the Nystr\"om  approximation error for $1000$ points uniformly sampled from a cube with edge length $10$,  when the Gaussian kernel function with length-scale $l = 10$ is used. FPS and random sampling are used to sample $k$ points from $X$ to form $X_k$. Nystr\"om error is computed only for the ranks which are multiples of $10$.}
    \label{fig:fill-distance-and-error temp}
\end{figure}

\begin{remark}
\label{rmk: lengthscale-plays-key-roles}
   The error estimate in \eqref{corollary:approximation fill distance 2-norm bound} does not involve the length-scale $l$ explicitly. However, this error estimate can still help understand how the length-scale in Gaussian kernels affects the Nystr\"om approximation error when the same landmark points $X_k$ are used. Assume $h_{X_k}$ is the fill distance of $X_k$ associated with the unit length-scale. When we change the length-scale to $l$, the kernel matrix associated with length-scale $l$ can be regarded as a kernel matrix associated with the unit length-scale and the scaled data points $\Tilde{\bx} = \bx/l$. This is because $\norm{\Tilde{\bx} - \Tilde{\by}} = \norm{\frac{\bx}{l} - \frac{\by}{l}} = \frac{1}{l}\operatorname{dist}(\bx,\by)$. In this case, the fill distance on the rescaled data points becomes $\frac{h_{X_k}}{l}$. As a result, as $l$ increases, the exponential factor in the estimate decays faster. This is consistent with the fact that the {Gaussian} kernel matrix $\bK$ is numerically low-rank when $l$ is large. 
\end{remark}

\subsection{FPS and Screening Effect}
\label{subsec:fps and screening}
In this section, we delve deeper into the connection between FPS and the screening effect, providing further justification for the integration of FPS within the \texttt{AFN} framework. Empirical evidence in \cite{guinness2018permutation} supports FPS's superiority over alternative methods in numerous scenarios. The screening effect's effectiveness, closely tied to the uniformity of sampled points as measured by the ratio $\delta_{X} = \frac{q_{X}}{h_{X}}$, requires $\delta_{X}$ to be strictly bounded, both lower and upper, a condition met by FPS as proven in Section \ref{subsec:fps}. Further insights from Section 3.2 in \cite{schafer_compression_2021} and \cite{schafer_sparse_2021} indicate that FPS's selection of evenly distributed points not only enhances the sparsity in the Cholesky decomposition of the kernel matrix but also, the precision matrix. Thus, FPS's application in sampling could increase the sparsity in $(\Schur)^{-1}$, improving the \texttt{AFN} preconditioner's efficiency and accuracy.
We now demonstrate the screening effect (mentioned in Section \ref{subsec:screening effect and fsai}) numerically with an example in Figure \ref{fig:sparsitySMWK11} when FPS is applied to select landmark points. Figure \ref{fig:sparsitySMWK11} shows histograms of the magnitude of the entries in three matrices $\bK_{22}+\mu \bI$, $\Schur$ and $(\Schur)^{-1}$ for $l=5$, with the matrices scaled so that their maximum entries are equal to one. The $1000$ data points $X$ are generated uniformly over a cube with edge length $10$ and $100$ landmark points $X_{100}$ are selected by FPS. The figure shows that $\Schur$ and its inverse have many more entries with smaller magnitude than $\bK_{22} + \mu\bI$. This example further justifies that $(\Schur)^{-1}$ has more ``sparsity" than $\Schur$, which supports the fact that FPS promotes sparsity in both kernel matrix and precision matrix and further supports the use of \texttt{FSAI} to $(\Schur)^{-1}$.
\begin{figure}[H]
   \begin{minipage}{0.32\textwidth}
      \centering
      \includegraphics[width=.99\linewidth]{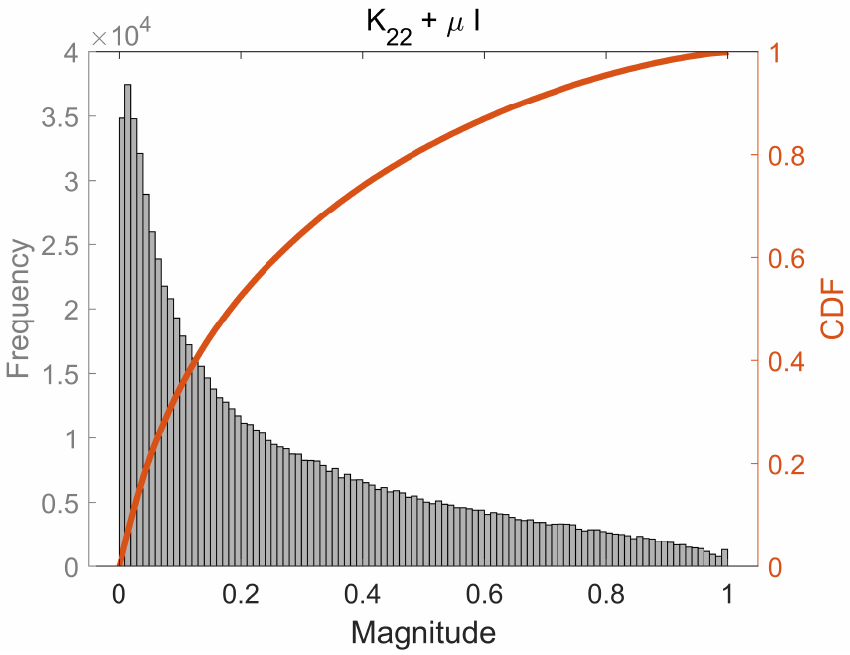}
   \end{minipage}
     \begin{minipage}{0.32\textwidth}
      \centering
      \includegraphics[width=.99\linewidth]{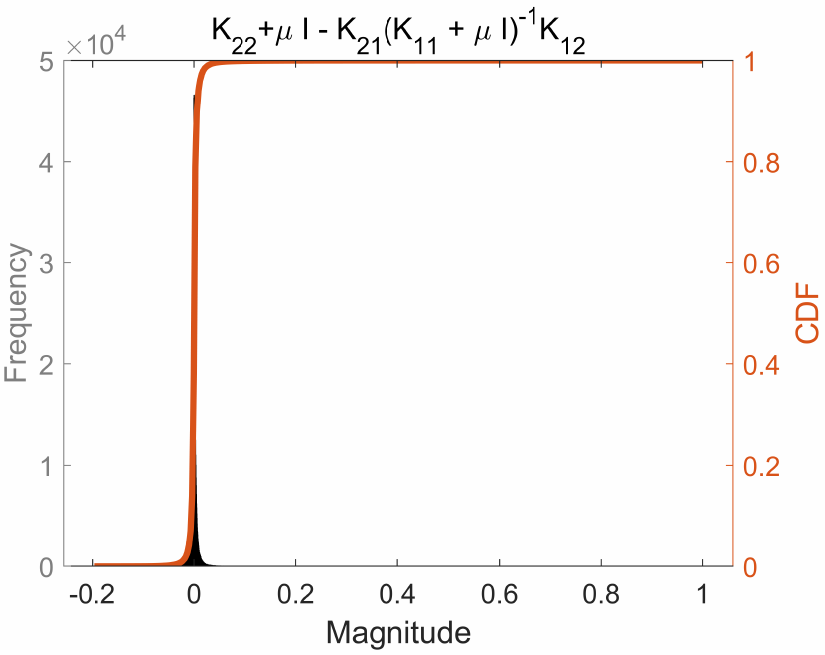}
   \end{minipage}
     \begin{minipage}{0.32\textwidth}
      \centering
      \includegraphics[width=.99\linewidth]{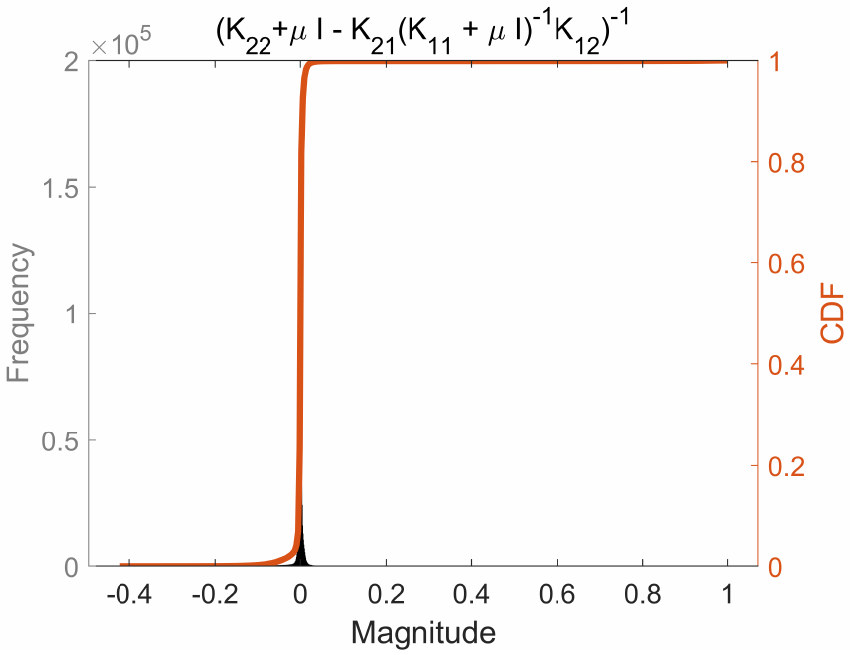}
   \end{minipage}
    \caption{Histograms of the magnitude of the entries in $\bK_{22}+\mu\bI$, $\Schur$, and  $(\Schur)^{-1}$ associated with a Gaussian kernel matrix defined using $1000$ points sampled uniformly from a cube with edge length $10$, regularization parameter $\mu=0.0001$, and length-scale $l=5$. The maximum entries in these three matrices are all scaled to $1$. $\bK$ has $243$ eigenvalues greater than $1.1\times\mu$.}
    \label{fig:sparsitySMWK11}
\end{figure}

\subsection{Implementation of FPS}
\label{subsec:implementation of fps}
A straightforward implementation of FPS for selecting \(k\) samples from \(n\) points in \(\mathbb{R}^d\) has a computational complexity of \(O(dk^2n)\). However, this complexity can be reduced to \(O(dkn)\) by maintaining a vector to track the distances between the unsampled points and the points already sampled. This optimization is elaborated in Algorithm \ref{alg:fps}.
\begin{algorithm}[H]
\caption{Farthest Point Sampling (\texttt{FPS})}
\begin{algorithmic}[1]
\STATE{\textbf{Input}: Dataset $X$ of size $n$, number of samples $k$}
\STATE{\textbf{Output}: Landmark point set $X_k$ of size $k$}
\STATE{Assign index to data points in $X$ as $\bx_1, \ldots, \bx_{n}$}
\STATE Calculate $\bar{\bx} = \frac{1}{n} \sum_{i=1}^{n} \bx_i$.
\STATE{Set $l=\underset{1\leq j\leq n}{arg\,min}\operatorname{dist}(\mathbf{x}_j,\bar{\mathbf{x}})$}
\STATE{Initialize the set $X_k=\{\mathbf{x}_l\}$}
\STATE{Initialize the distance vector $\mathbf{d}$ of length $n$, setting all entries to $\inf$}
\FOR{$i = 1$ to $k-1$}
    \STATE Update $\mathbf{d}$ with entries $\mathbf{\mathbf{d}}(j) = \min\{\mathbf{d}(j), \text{dist}(\bx_j, \bx_l)\}$
    \STATE Choose $l = \underset{1\leq j\leq n}{arg\,max}\mathbf{d}(j)$
    \STATE{Add $\mathbf{x}_l$ to $X_k$}
\ENDFOR
\STATE{\textbf{Return}: $X_k$}
\end{algorithmic}
\label{alg:fps}
\end{algorithm}

\section{Numerical experiments}
\label{sec:experiment}
The \texttt{AFN} preconditioner and the preconditioning strategy (Algorithm
\ref{alg:strategy}) are tested for the iterative solution of regularized
kernel matrix systems \eqref{eq:Problem} over a wide range of length-scale
parameters $l$ in the following two kernel functions
\begin{itemize}
    \item Gaussian kernel: $\mathcal{K}(\bx,\by) = \exp\left(-\frac{1}{l^2}\Vert\bx-\by\Vert_2^2\right)$
    \item Mat\'ern-3/2 kernel: $\mathcal{K}(\bx,\by) = \left(1 + \frac{\sqrt{3}}{l}\vert| \bx - \by \vert|_2\right) \exp\left(-\frac{\sqrt{3}}{l} \vert| \bx - \by \vert|_2\right)$.
\end{itemize}


We also benchmark the solution of these systems using unpreconditioned \texttt{CG},
and preconditioned \texttt{CG}, with the \texttt{FSAI} preconditioner and with the 
randomized Nystr\"om (\texttt{RAN}) preconditioner
\cite{frangella_randomized_2021} with randomly selected $k$ landmark points.

\texttt{RAN} approximates
the kernel matrix with a rank-$k$ Nystr\"om approximation based on 
randomly sampling the data points.
Assuming the $k$-th largest eigenvalue of $\Knys$ is $\lambda_k$, the inverse
of the \texttt{RAN} preconditioner takes the form \cite{frangella_randomized_2021}:
$(\lambda_k+\mu)\bU(\mathbf{\Lambda}+\mu\bI)^{-1}\bU^\top+(\bI-\bU\bU^\top)$
where  $\bU\mathbf{\Lambda}\bU^\top$ is the eigendecomposition of
$\Knys$.
{{In our experiments, we use $400$ nearest neighbors as the sparsity pattern for \texttt{FSAI}, fix the Nystr\"om rank to be $3000$ for \texttt{RAN}, and use $100$ nearest neighbors as the sparsity pattern for the \texttt{FSAI} used in \texttt{AFN}. }}

The stopping tolerance for the relative residual norm is set to be $10^{-4}$.
We randomly generated right-hand side vectors in Equation \eqref{eq:Problem}
with entries from the uniform distribution $[-0.5, 0.5]$.
For all tests we perform 3 runs and report the average results.

{\texttt{AFN}, \texttt{RAN} and \texttt{FSAI} have been implemented in C.} {The C implementation of the  \texttt{AFN} preconditioner can be found in the AFN\_Precond branch of the H2Pack GitHub website \footnote{\href{https://github.com/scalable-matrix/H2Pack/tree/AFN\_precond}{https://github.com/scalable-matrix/H2Pack/}}.
    The test routines for \texttt{AFN} and \texttt{RAN} can be found from this web page \footnote{\href{https://github.com/scalable-matrix/H2Pack/tree/AFN_precond/examples/AFN_precond}{https://github.com/scalable-matrix/H2Pack/tree/AFN\_precond/examples/AFN\_precond}} and the test routines for \texttt{FSAI} can be found from this web page \footnote{\href{https://github.com/scalable-matrix/H2Pack/tree/AFN_precond/examples/SPDHSS-H2}{https://github.com/scalable-matrix/H2Pack/tree/AFN\_precond/examples/SPDHSS-H2}}.} Experiments are run on an Ubuntu 20.04.4 LTS machine equipped with 755
GB of system memory and a 24-core 3.0 GHz Intel Xeon Gold 6248R CPU.
We build our code with the GCC 9.4.0 compiler and take advantage of
shared memory parallelism using OpenMP. We use the parallel \texttt{BLAS}
and \texttt{LAPACK} implementation in the \texttt{OpenBLAS} library for
basic matrix operations. \texttt{H2Pack} \cite{DDSMASH,10.1145/3412850}
is used to provide linear complexity matrix-vector multiplications
associated with large-scale $\bK$ for 3D datasets with the relative error threshold $10^{-8}$. {We utilized a brute force parallel FPS algorithm on the global dataset. OpenMP was used to apply an $O(n)$ distance update in parallel at each step. The computational cost is tractable due to a maximum of $2000$ distance updates required. 
} The number of OpenMP
threads is set to 24 in all the experiments.

\subsection{Experiments with synthetic 3D datasets} 

The synthetic data consists of $n=1.6\times 10^{5}$ random points sampled
uniformly from inside a 3D cube with edge length $\sqrt[3]{n}$.  We first solve
regularized linear systems associated with both Gaussian kernel and Mat\'ern-3/2 kernel, with $\mu = 0.0001$.

{The computational results are tabulated in Table \ref{tab:synthetic-all},
which shows the number of solver iterations required for convergence, the preconditioner setup (construction) time, and the time required
for the iterative solve. Rank estimation Algorithm \ref{alg:nys-estimation} is used to estimate the rank $k$ for each kernel matrix with the given length-scale information shown on the first row of each table. For both kernels,
we select 9 \emph{middle length-scales} to justify the robustness of \AFN. We also include two extreme length-scales in these tables to show the effectiveness of the preconditioning strategy using \AFN\ summarized in Algorithm \ref{alg:strategy} across a wide range of $l$.} 

We first note that, for unpreconditioned \texttt{CG},
the iteration counts first increase and then decrease as the length-scale decreases for both kernel functions. This confirms the result seen earlier in
Figure \ref{fig:varyingl} that it is the linear systems associated with
the \emph{middle length-scales} that are most difficult to solve due to the
unfavorable spectrum of these kernel matrices. We also observe that \texttt{FSAI} is very effective as a preconditioner for Gaussian kernel, with $l^2 = 0.1$ and Mat\'ern-3/2 kernel, with $l = 1.0$. \texttt{FSAI} is effective if the inverse of the kernel matrix can be approximated by a sparse matrix, which is the situation for both length-scales. We observe the opposite effect for the \texttt{RAN} preconditioner, which is
effective for large length-scales but is poor for small length-scales. 
For middle length-scales, \texttt{AFN} substantially reduces the number of iterations compared to other methods. In particular, \texttt{AFN} yields almost a constant iteration number for Mat\'ern-3/2 kernel. For Gaussian kernel with $l^2=1000$ and Mat\'ern-3/2 kernel with $l=1000$, choosing \AFN\ as the Nystr\"om preconditioner form with the estimated rank significantly reduces the setup time for \AFN\ compared to \texttt{RAN}(3000) but still keeps roughly the same preconditioning effect.

\begin{table}[htp]
     \caption{Numerical results for the kernel matrices defined based on $n=1.6\times 10^{5}$ points sampled inside a 3D cube of edge length $\sqrt[3]{n}$. $``-"$ indicates that a run failed to converge within $500$ iterations. All experiments are run three times and reported as the average of three runs.} 
    \label{tab:synthetic-all}
     \tabcolsep2.8pt
\footnotesize
\begin{subtable}[h]{.98\textwidth}
\centering
\tabcolsep3.6pt
\begin{tabular}{c|ccccccccccc}
\hline  
$l^2$ & $1000$ & $65$ & $60$ & $55$ & $50$ & $45$& $40$ & $35$ & $30$ & $25$ & $0.1$\\
\hline
$k$ & $565$ & $9600$ & $9600$ & $9600$ & $9600$ & $12800$ & $12800$ & $12800$ & $16000$ & $19200$ & $160000$\\
\hline
\multicolumn{11}{c}{Iteration Counts} \\
\hline 
\texttt{CG} & 44.00 & - & - & - & - & - & - & - & - & - & 1.00 \\
\texttt{AFN} &  3.00 & 35.00 & 37.00 & 38.00 & 40.00 & 42.00 & 46.00 & 50.00 & 57.00 & 62.00 & 1.00 \\
\texttt{RAN} &  3.00  &72.67 & 101.33 & 140.67 & 199.33 & 284.33 & 409.33 & - & - & - & - \\
\texttt{FSAI} & - & - & - & - & - & - & - & - & - & - & 1.00   \\
\hline
\multicolumn{11}{c}{Setup Time (s)} \\
\hline
\texttt{AFN}  & 3.19 & 38.97 & 39.75 & 40.10 & 39.73 & 39.89 & 40.76 & 39.34 & 40.12 & 40.59 & 40.37\\
\texttt{RAN} & 27.28 & 27.59 & 26.46 & 27.33 & 29.05 & 29.95 & 31.18 & 31.56 & 33.64 & 33.97 & 35.07\\
\texttt{FSAI} & 10.00 & 9.91 & 10.02 & 10.16 & 9.72 & 9.87 & 10.14 & 9.71 & 10.01 & 9.84 & 13.22 \\
\hline
\multicolumn{11}{c}{Solve Time (s)} \\
\hline
\texttt{CG} & 9.72 & - & - & - & - & - & - & - & - & - & 1.75 \\
\texttt{AFN}  & 0.43 & 12.49 & 14.00 & 14.99 & 15.82 & 18.02 & 20.15 & 22.59 & 27.26 & 29.10 & 1.91 \\
\texttt{RAN} & 0.81 &  23.29 & 35.73 & 49.98 & 72.20 & 96.75 & 138.88& - & - & - & - \\
\texttt{FSAI}  & - & - & - & - & - & - & - & - & - & - & 1.27\\
\hline
\hline
    \end{tabular}
    \caption{Gaussian kernel with a fixed $\mu=0.0001$ and varying $l$. }
    \end{subtable}
\begin{subtable}[h]{.98\textwidth}
\centering
\begin{tabular}{c|ccccccccccc}
\hline  
$1/l$ &  1.0 & 0.065 & 0.060 & 0.055 & 0.050 & 0.045 & 0.040 & 0.035 & 0.030 & 0.025 & 0.001\\
\hline
$k$ &  160000 & 19200 & 16000 & 14080 & 12800 & 9600 & 9600 & 6400 & 6400 & 6400 & 178\\
\hline
\multicolumn{11}{c}{Iteration Counts} \\
\hline    
\texttt{CG} & 293.67 & - & - & - & - & - & - & - & - & - & 292.67 \\
\texttt{AFN} & 3.00 & 6.00 & 6.00 & 6.00 & 7.00 & 7.00 & 7.00 & 7.00 & 7.00 & 6.00 &  9.00   \\
\texttt{RAN} & -  & 454.00 & 404.33 & 355.67 & 308.33 & 263.00 & 220.67 & 181.00 & 142.00 & 108.33 & 4.00  \\
\texttt{FSAI} & 5.00 & - & - & - & - & - & - & - & - & - & - \\
\hline
\multicolumn{11}{c}{Setup Time (s)} \\
\hline 
        
\texttt{AFN} & 47.32 & 45.24 & 44.67 & 42.99 & 43.41 & 43.39 & 44.34 & 43.50 & 43.29 & 42.74  & 3.07  \\

\texttt{RAN} & 63.69 & 39.78 & 40.30 & 39.81 & 40.16 & 39.94 & 40.08 & 40.19 & 40.18 & 39.77 & 55.41 \\
\texttt{FSAI} & 13.98  & 10.31 & 10.18 & 10.19 & 10.29 & 10.26 & 10.30 & 10.28 & 10.02 & 9.84 & 13.80  \\
\hline
\multicolumn{11}{c}{Solve Time (s)} \\
\hline 
\texttt{CG} & 22.41 & - & - & - & - & - & - & - & - & - & 22.40 \\
\texttt{AFN} &  2.43 & 2.52 & 2.63 & 2.42 & 3.32 & 2.84 & 3.02 & 2.58 & 2.74 & 2.30 &  0.86 \\

\texttt{RAN} &  - & 116.37 & 99.32 & 86.87 & 74.04 & 63.98 & 53.58 & 42.24 & 32.19 & 25.93 & 1.36\\
\texttt{FSAI} &  3.71 & - & - & - & - & - & - & - & - & - & - \\
\hline
\hline
    \end{tabular}
    \caption{\matern kernel with a fixed $\mu=0.0001$ and varying $l$.}
    \end{subtable}
\end{table}

In Table \ref{tab:synthetic-mu}, we also compare the performance of \texttt{AFN}, \texttt{RAN} and \texttt{FSAI} for solving \eqref{eq:Problem} associated with the \matern kernel matrices with $l=20$ and varying $\mu$. It is easy to see that the performance of \texttt{RAN} and \texttt{FSAI}
deteriorates as the regularization parameter $\mu$ decreases while the iteration count of \texttt{AFN} remains almost a constant, which shows the improved robustness of \texttt{AFN} over \texttt{RAN} and \texttt{FSAI} with respect to $\mu$.

\begin{table}[htp]
     \caption{Numerical results for the \matern kernel matrices associated with $l=20$ and varying $\mu$ and $n=1.6\times 10^{5}$ points sampled inside a 3D cube of edge length $\sqrt[3]{n}$. $``-"$ indicates that a run failed to converge within $500$ iterations. All experiments are run three times and reported as the average of three runs.} 
    \label{tab:synthetic-mu}
     \tabcolsep2.8pt
\footnotesize
\centering
\tabcolsep5pt
\begin{tabular}{c|cccccccccc}
\hline  
$\mu$ &  $1e\!-\!1$ & $1e\!-\!2$ & $1e\!-\!3$ & $1e\!-\!4$ & $1e\!-\!5$ & $1e\!-\!6$ & $1e\!-\!7$ & $1e\!-\!8$ & $1e\!-\!9$ & $1e\!-\!10$ \\
\hline
\multicolumn{11}{c}{Iteration Counts} \\
\hline    
\texttt{CG} & - & - & - & - & - & - & - & - & - & -  \\
\texttt{AFN} &  15.00 & 12.00 & 6.00 & 7.00 & 7.00 & 7.00 & 7.00 & 7.00 & 7.00 & 7.00    \\
\texttt{RAN} & 10.33 & 29.00 & 93.33 & 311.33 & - & - & - & - & - & -   \\
\texttt{FSAI} & 164.00 & 370.33  & - & - & - & - & - & - & - & -  \\
\hline
\multicolumn{11}{c}{Setup Time (s)} \\
\hline 
        
\texttt{AFN} & 43.74 & 43.50 & 42.74 & 44.59 & 43.63 & 43.24 & 44.31 & 44.30 & 43.11 & 43.71    \\
\texttt{RAN} &     40.25 & 39.71 & 39.14 & 40.86 & 39.92 & 40.13 & 40.40 & 40.34 & 39.80 & 40.35 \\
\texttt{FSAI} &  10.33 & 10.46 & 10.56 & 10.39 & 10.53 & 10.40 & 10.53 & 10.59 & 10.76 & 10.48 \\
\hline
\multicolumn{11}{c}{Solve Time (s)} \\
\hline 
\texttt{CG} &- & - & - & - & - & - & - & - & - & - \\
\texttt{AFN} &    5.30 & 4.95 & 2.61 & 2.78 & 3.02 & 2.90 & 2.89 & 2.84 & 2.88 & 3.09 \\
\texttt{RAN} &   3.29 & 8.53 & 25.03 & 76.33 &   - & - & - & - & - & -\\
\texttt{FSAI} & 21.43 & 46.44 & - & - & - & - & - & - & - & - \\
\hline
\hline
    \end{tabular}
\end{table}

\subsection{Experiments with machine learning datasets}



{We test the performance of \texttt{AFN} on two high-dimensional datasets, namely \texttt{IJCNN1} from LIBSVM \cite{chang2011libsvm} and \texttt{Elevators} from UCI \cite{dua2017uci} in this section. The training set of \texttt{IJCNN1} consists of $n=49990$ data points, with $22$ features and $2$ classes, while \texttt{Elevators} contains $n=16599$ data points, with $18$ features and $1$ target.}  

{
Here, we perform experiments with the Gaussian kernel for \texttt{IJCNN1} and \matern kernel for \texttt{Elevators}. After conducting grid searches, we select the regularization parameter to be $\mu = n\times 10^{-6}$ for both datasets so that the test error of KRR is small for the optimal length-scale $l$ in our searches.
We select $12$ length-scales {in two separate intervals}, which include the optimal length-scales for both datasets. The grid search method was used to determine the optimal length-scale for \texttt{IJCNN1}, resulting in a value of $l=1$ which is consistent with the findings in \cite{frangella_randomized_2021}. In contrast, for \texttt{Elevators}, the optimal length-scale was determined using GPyTorch \cite{wenger2022preconditioning} and found to be $l = 14$. 
Most of the length-scales within each interval correspond to middle length-scales. Two extreme length-scales are also considered here to show the effectiveness of \texttt{AFN} across a wide range of $l$. Since \texttt{FSAI} is less robust than \texttt{RAN}, we only compare \texttt{AFN} with \texttt{RAN} in this section.
As these are moderate-dimensional datasets (22 and 18 dimensions, as mentioned) and we do not have a fast kernel matrix-vector multiplication code for these datasets, the kernel matrix-vector multiplications were performed explicitly. Due to the high computational cost of FPS in high dimensions, we simply use uniform sampling to select the landmark points for \texttt{AFN} when the estimated rank is greater than $2000$ in these experiments.

\begin{table}[htp]
         \caption{Numerical results for the \texttt{IJCNN1} and \texttt{Elevator} datasets with Gaussian kernel and \matern kernel, respectively. $``-"$ indicates that a run failed to converge within $500$ iterations. All experiments are run three times and reported as the average of three runs. {In both tests we set $\mu = n\times 10^{-6}$.}}
     \label{tab:real-all}
     \tabcolsep1.8pt
\footnotesize
\begin{subtable}[h]{.98\textwidth}
\tabcolsep2.5pt
\centering
\begin{tabular}{c|cccccccccccc}
\hline  
$l^2$ & 10.0 & 1.0 & 0.9 & 0.8 & 0.7 & 0.6 & 0.5 & 0.4 & 0.3 & 0.2 & 0.1 & 0.01 \\
\hline
$k$ & 1278 & 8798 & 10397 & 11197 & 13197 & 14996 & 17396 & 20395 & 24394 & 29394 & 37192 & 48190\\
\hline
\multicolumn{13}{c}{Iteration Counts} \\
\hline 
\texttt{CG}  & 218.00   &- & - & - & - & - & - & - & - & 481.00 & 418.00  & 239.00  \\
\texttt{AFN}  & 3.00 & 44.00 & 43.33 & 42.00 & 41.00 & 39.00 & 36.67 & 33.00 & 29.33 & 25.33 & 19.67 & 9.00\\
\texttt{RAN}  & 2.00  & 12.67 & 13.67 & 15.67 & 18.67 & 21.67 & 26.00 & 32.00 & 40.00 & 51.00 & 66.67  & 73.33 \\
\hline
\multicolumn{13}{c}{Setup Time (s)} \\
\hline
\texttt{AFN}  & 4.18 & 15.69 & 15.66 & 15.30 & 15.53 & 15.29 & 15.30 & 15.68 & 16.34 & 15.51 & 15.19 & 15.15 \\
\texttt{RAN} & 52.44  & 40.81 & 41.68 & 41.20 & 41.73 & 41.40 & 41.09 & 41.59 & 41.08 & 40.90 & 43.58  & 48.16 \\
\hline
\multicolumn{13}{c}{Solve Time (s)} \\
\hline
\texttt{CG}  & 30.63  &  - & - & - & - & - & - & - & - & 55.23 & 46.73  & 34.73 \\   
\texttt{AFN} & 0.97 & 8.07 & 8.99 & 8.24 & 7.47 & 7.55 & 6.88 & 6.50 & 5.94 & 5.05 & 5.01 & 2.44 \\
\texttt{RAN}  & 0.70 & 2.93 & 3.04 & 3.01 & 4.03 & 4.90 & 4.87 & 6.13 & 8.11 & 9.40 & 11.89  & 12.83 \\
\hline
\hline
    \end{tabular}
   \caption{\texttt{IJCNN1} with Gaussian kernel.}
    \label{tab:ijcnn-all}
    \end{subtable}

\begin{subtable}[h]{.98\textwidth}
\centering
\tabcolsep1.7pt
\begin{tabular}{c|cccccccccccc}
\hline  
$1/l$ & 1.0 &0.1  &  0.09  &  0.08  &  0.07  &  0.06  &  0.05  &  0.04  &  0.03  &  0.02  &  0.01 & 0.0005\\
\hline
 $k$ & 16599 & 12083 &  11685 & 11419 & 11087 & 10822 & 10224 & 9427 & 8166 & 6838 & 5576 & 983 \\
\hline
\multicolumn{13}{c}{Iteration Counts} \\
\hline   
\texttt{CG} & 29.00 & 324.00 & 325.00 & 331.00 & 339.00 & 347.00 & 355.00 & 358.00 & 349.00 & 331.00 & 303.00 & 124.00\\
\texttt{AFN}  & 3.00 & 9.33 & 9.67 & 9.67 & 10.00 & 10.00 & 10.00 & 10.00 & 10.00 & 49.00 & 60.00 & 5.00 \\
\texttt{RAN} & 20.67 & 71.67 & 71.00 & 69.33 & 67.00 & 65.00 & 61.00 & 57.33 & 59.67 & 69.67 & 75.33 & 7.33 \\
\hline
\multicolumn{13}{c}{Setup Time (s)} \\
\hline   
\texttt{AFN} & 9.58 & 5.34 & 5.45 & 5.79 & 5.60 & 5.48 & 5.42 & 5.47 & 5.36 & 5.76 & 6.06 & 1.94 \\
\texttt{RAN} & 38.78 & 28.64 & 44.28 & 42.45 & 30.86 & 32.53 & 44.61 & 36.91 & 39.38 & 38.32 & 35.72 & 34.90 \\
\hline
\multicolumn{13}{c}{Solve Time (s)} \\
\hline 
\texttt{CG} & 0.54 & 3.65 & 3.73 & 3.71 & 3.79 & 3.92 & 4.01 & 4.06 & 3.93 & 3.75 & 3.48 & 1.39\\
\texttt{AFN} & 0.21 & 0.38 & 0.40 & 0.43 & 0.40 & 0.40 & 0.49 & 0.39 & 0.38 & 1.83 & 2.22 & 0.11 \\
\texttt{RAN} & 0.68 & 2.04 & 1.84 & 2.08 & 1.82 & 1.76 & 1.67 & 1.49 & 1.76 & 1.88 & 2.00 & 0.28 \\
\hline
\hline
    \end{tabular}
     \caption{\texttt{Elevators} with \matern kernel.}
     \label{tab:elevator-all}
    \end{subtable}
\end{table}
We report the computational results in Table \ref{tab:real-all}. 
The patterns of the change of iteration counts, setup time and solution time with respect to the length-scales on both datasets are similar to those observed in the 3D experiments. 
First, the iteration counts of unpreconditioned \texttt{CG} first increases and then decreases as $l$ decreases in both datasets. This indicates that the spectrum of the kernel matrices associated with high-dimensional datasets {could be related to} those associated with low-dimensional data. 
\texttt{AFN} is {again} able to significantly reduce the iteration counts compared to unpreconditioned \texttt{CG} in all tests.
{We notice that the iteration count of the \texttt{RAN} preconditioned \texttt{CG} increases as the estimated rank increases on the \texttt{IJCNN1} dataset. This implies that in order to converge in the same number of iterations as $l$ becomes smaller, \texttt{RAN} type preconditioners need to keep increasing the Nystr\"om approximation rank $k$ and thus require longer setup time and more storage. $\texttt{AFN}$ requires smaller setup time in all of the experiments and leads to smaller iteration counts when $l^2<0.4$ on the \texttt{IJCNN1} dataset and all length-scales on the \texttt{Elevators} dataset. In addition,
we can also observe that \texttt{AFN} yields the smallest total time in all of the experiments on both datasets compared with $\texttt{RAN}$.

\if
Since it is also observed that the performance gap between \AFN\, and \texttt{FSAI}, is very small for large-rank matrices when \texttt{FSAI}, performs best while \AFN\, consistently outperforms \texttt{FSAI}, for small-rank and middle-rank matrices, so \texttt{FSAI}, will also not be reported. Considering only \RAN\, may be better than \AFN\, for small-rank matrices in the sense of fewer construction costs and less setup time, we only test \AFN\, against \RAN\, to justify the robustness and superiority of \AFN\, for middle-rank matrices. For both IJCNN and Sensorless datasets, \AFN\, performs consistently well. All metrics of it are very stable through all the length-scales. However, with the decrease of the lengthscales, iteration counts of \RAN\, keep increasing. Especially, the setup time of \RAN\, is roughly 2.5 times larger than that of \AFN. For both kernels, \AFN\, always has a smaller total time compared to \RAN. 
}
\fi

\section{Conclusion}
\label{sec:conclusion}

In this paper, we introduced an approximate block factorization of
$\bK + \mu \bI$
that is inspired by the existence of a Nystr\"om approximation,
$\bK \approx \bK_{X,X_k} \bK_{X_k,X_k}^{-1} \bK_{X_k,X}$.
The approximation is designed to efficiently handle the case where
$k$ is large, by using sparse approximate inverses.

We further introduced a preconditioning strategy that is robust
for a wide range of length-scales. When the length-scale is large, existing Nystr\"om
preconditioners work well. For the challenging length-scales, the \texttt{AFN}
preconditioner proposed in this paper is the most effective. We justify the use of FPS to select landmark points in order to construct an accurate and stable \texttt{AFN} preconditioner and propose a rank estimation algorithm using a subsampling of the entire dataset.

\shifanrewriting{It is important to note that in high-dimensional settings, the effectiveness of screening effects diminishes, as indicated by \cite{schafer_sparse_2021,schafer_compression_2021}. This is attributed to the reduced representational capacity of Euclidean distance for spatial similarity in high-dimensional spaces, a concept further explored by \cite{domingos2012few}. Consequently, the \texttt{FSAI} approach for approximating the inverse of the Schur complement can be less effective for high-dimensional datasets, such as those commonly found in machine learning, as it is for lower-dimensional ones, such as those in spatial statistics. Nevertheless, in the realm of machine learning, kernel methods -- including the kernel trick in Support Vector Machines (SVMs), Kernel Ridge Regression (KRR), and Gaussian Process Regression (GPR) -- fundamentally rely on the premise that spatial similarity correlates with data similarity and the proposed \texttt{AFN} method retains its relevance as long as this assumption is valid. For datasets with  high dimensionality, we plan to first apply a transformation to map the data points to lower-dimensional manifolds. This transformation, as discussed in the survey  \cite{binois2022survey}, ensures that Euclidean distance continues to effectively represent similarity in these reduced-dimensional spaces. In future work, we will also study whether the dependence on ambient
dimension in Theorem \ref{thm:FPS is near optimal} can be reduced to
the intrinsic dimension of the data manifold and apply \texttt{AFN} to accelerate the convergence of stochastic trace estimation and gradient based optimization algorithms. 
}


\appendix

\section{Proof of Theorem 4.1}
\label{sec:fast_fps}

{The proof of Theorem \ref{corollary:approximation fill distance 2-norm bound} relies on Theorem \ref{thm:interpolation-bound} from \cite{belkin_approximation_2018}. Theorem \ref{thm:interpolation-bound} states that any bounded map $\mathcal{T}$ from a Hilbert space to a RKHS $\mathcal{H}$ corresponding to certain smooth radial kernels {such as the Gaussian kernel defined in \eqref{eq:gaussianf} and the inverse multiquadrics kernel defined in \eqref{eq:inverseMultiquadric}} always admits a low rank approximation in {$L_{\mu}^{2} := \{f(x) | \int |f(x)|^2d\mu < \infty \}$.} Furthermore, the approximation error bound can be quantified by fill distance.}
Before we proceed to Theorem \ref{thm:interpolation-bound}, we first introduce a few notations that will be used in the statement of Theorem \ref{thm:interpolation-bound}. On a domain $\Omega$, the integral operator $\mathcal{K}_{\mu}: L_{\mu}^{2} \rightarrow \mathcal{H}$ is defined as:
\begin{align*}
   \mathcal{K}_{\mu}(f)(\cdot) = \int \mathcal{K}(\cdot, \bx)f(\cdot)d\mu.
\end{align*}
The restriction operator $\mathcal{R}_{\mu} : \mathcal{H} \rightarrow L_{\mu}^{2}$ is defined as the restriction of $f\in \mathcal{H}$ to the support of $\mu$, interpolation operator $\mathcal{S}_{X_k}: \mathcal{H} \rightarrow \mathcal{H}$ is defined by interpolating the values of $f$ on a subset $X_k\subset \Omega$ as:
$$
\mathcal{S}_{X_k}(f)(\bx) = \sum_{i=1}^{k}\alpha_i\mathcal{K}(\bx_i,\bx),$$ 
with $(\alpha_1,\dots,\alpha_k)^{\top} = \bK_{X_k,X_k}^{-1}(f(\mathbf{x}_1),\dots, f(\mathbf{x}_k))^{\top}.$
Since the range of $\mathcal{R}_{\mu}$ and $\mathcal{S}_{X_k}$ is different, the following norm is used to measure their difference:
$$\|\mathcal{R}_{\mu}-\mathcal{S}_{X_k}\|_{\mathcal{H}\rightarrow L_{\mu}^{2}} := \max_{f\in \mathcal{H},f\neq 0}\frac{\|(\mathcal{R}_{\mu}-\mathcal{S}_{X_k})(f)\|_{L_{\mu}^2}}{\|f\|_{\mathcal{H}}}.$$
\begin{theorem}[\cite{belkin_approximation_2018}]
\label{thm:interpolation-bound}
Let $\mathcal{H}$ denote the RKHS corresponding to the kernel $\mathcal{K}$. Given a probability measure $\mu$ on $\Omega$ and a set $X_k\subset \Omega$, there exist constants $C',~C''>0$ such that
\begin{equation}
    \|\mathcal{R}_{\mu} - \mathcal{S}_{X_k}\|_{\mathcal{H}\rightarrow L_{\mu}^{2}} < C'\exp(-C''/h_{X_k}).
\end{equation}
\end{theorem}
When $\Omega=X = \{\bx_1, \dots, \bx_n\}$ and the uniform discrete measure $\mu_{{X}} =\frac{1}{n} \sum_{i=1}^{n} \delta_{\bx_i}$ is used with $\delta_{\bx_i}$ being the Dirac measure at point $\bx_i$, we have
$$
\mathcal{K}_{\mu_{{X}}}(f)(\bx) = \frac{1}{n}\sum_{i=1}^{n} \mathcal{K}
(\bx_{i},\bx)f(\bx_{i})$$
and $\mathcal{H}_{{X}}= \operatorname{span}\{K(\bx_1,\cdot),\dots,K(\bx_n,\cdot)\}.$ 
The integral operator, interpolation operator and restriction operator can then be written in the matrix form as $\mathcal{K}_{\mu_{{X}}}(f)(X) =  \frac{1}{n}\mathbf{K}f(X)$,  $\mathcal{S}_{X_k}(f)({X}) = \bK_{X,X_k}\bK_{X_k,X_k}^{-1}{f}({X_k})$, and $\mathcal{R}_{\mu_{{X}}} = \bI\in \mathbb{R}^{n\times n}$, respectively.
Since $\mathcal{R}_{\mu_{X}}\circ \mathcal{K}_{\mu_{{X}}}=\mathcal{K}_{\mu_{{X}}}$, we have $$\mathcal{K}_{\mu_{{X}}}-\mathcal{S}_{{X}_k}\circ \mathcal{K}_{\mu_{{X}}}=(\mathcal{R}_{\mu_{X}} - \mathcal{S}_{X_k})\circ \mathcal{K}_{\mu_{{X}}}.$$
Thus, we can get the following inequality
\begin{align*}
     \|(\mathcal{R}_{\mu_{X}} - \mathcal{S}_{X_k})\circ \mathcal{K}_{\mu_{{X}}}\|_{L_{\mu_{{X}}}^{2}\rightarrow L_{\mu_{{X}}}^{2}}\leq  \|(\mathcal{R}_{\mu_{X}} - \mathcal{S}_{X_k})\|_{\mathcal{H}_{{X}}\rightarrow L_{\mu_{{X}}}^{2}} \|\mathcal{K}_{\mu_{X}}\|_{ L_{\mu_{{X}}}^{2} \rightarrow \mathcal{H}_{{X}}}.
\end{align*}
Based on Theorem \ref{thm:interpolation-bound}, we know that 
\begin{align*}
     \|(\mathcal{R}_{\mu_{X}} - \mathcal{S}_{X_k})\circ \mathcal{K}_{\mu_{{X}}}\|_{L_{\mu_{{X}}}^{2}\rightarrow L_{\mu_{{X}}}^{2}}\leq  C'\exp(-C''/h_{X_k}) \|\mathcal{K}_{\mu_{X}}\|_{ L_{\mu_{{X}}}^{2} \rightarrow \mathcal{H}_{{X}}}.
\end{align*}
In the next theorem, we will derive an error estimate for the Nystr\"om approximation error by further proving
\begin{align*}
\|(\mathcal{R}_{\mu_{X}} - \mathcal{S}_{X_k})\circ \mathcal{K}_{\mu_{{X}}}\|_{L_{\mu_{X}}^2\rightarrow L_{\mu_{{X}}}^{2}} &= \frac{1}{n}\Vert \bK- \Knys\Vert, 
\end{align*}
and $ \|\mathcal{K}_{\mu_{X}}\|_{L_{\mu_{{X}}}^{2}\rightarrow \mathcal{H}_{{X}}}^2 = \sqrt{\mathbf{\Vert K\Vert}/n}$.

\begin{manualtheorem}{4.1}The Nystr\"om approximation $\Knys = \bK_{{X},X_k} \bK_{X_k,X_k}^{-1}\bK_{X_k,X}$ to $\bK$ using the landmark points $X_{k} = \{\bx_{k_i}\}_{i=1}^{k}$ has the following error estimate
\begin{equation}
   \Vert \bK- \Knys\Vert < \sqrt{n\Vert \bK \Vert}  C'\exp(-C''/h_{X_k}),
\end{equation}
where $C'$ and $C^{''}$ are constants independent of $X_k$.
\end{manualtheorem}
\begin{proof}
Since $\mathcal{R}_{\mu_{X}}\circ \mathcal{K}_{\mu_{{X}}}=\mathcal{K}_{\mu_{{X}}}$, we have $$\mathcal{K}_{\mu_{{X}}}-\mathcal{S}_{{X}_k}\circ \mathcal{K}_{\mu_{{X}}}=(\mathcal{R}_{\mu_{X}} - \mathcal{S}_{X_k})\circ \mathcal{K}_{\mu_{{X}}}.$$
Notice $\mathcal{K}_{\mu_{{X}}}$ is a map from $L_{\mu_{{X}}}^{2}$ to $\mathcal{H}_{{X}}$ and from the definition of the norm, we get the following inequality
\begin{equation}
    \|(\mathcal{R}_{\mu_{X}} - \mathcal{S}_{X_k})\circ \mathcal{K}_{\mu_{{X}}}\|_{L_{\mu_{{X}}}^{2}\rightarrow L_{\mu_{{X}}}^{2}}\leq  \|(\mathcal{R}_{\mu_{X}} - \mathcal{S}_{X_k})\|_{\mathcal{H}_{{X}}\rightarrow L_{\mu_{{X}}}^{2}} \|\mathcal{K}_{\mu_{X}}\|_{ L_{\mu_{{X}}}^{2} \rightarrow \mathcal{H}_{{X}}}.
\end{equation}

Based on Theorem \ref{thm:interpolation-bound}, we obtain 
\begin{equation}
 \|(\mathcal{R}_{\mu_{X}} - \mathcal{S}_{X_k})\|_{\mathcal{H}_{{X}}\rightarrow L_{\mu_{{X}}}^{2}} < C'\exp(-C''/h_{X_k}).
\end{equation}


First, recall that 
\begin{align*}
\|(\mathcal{R}_{\mu_{X}} - \mathcal{S}_{X_k})\circ \mathcal{K}_{\mu_{{X}}}\|_{ L_{\mu_{{X}}}^{2}\rightarrow L_{\mu_{{X}}}^{2}} &= \max_{f\in  L_{\mu_{{X}}}^{2},f\neq 0}\frac{\|(\mathcal{R}_{\mu_{{X}}} - \mathcal{S}_{X_k})\circ \mathcal{K}_{\mu_{{X}}}(f)\|_{L_{\mu_{X}}^2}}{\|f\|_{ L_{\mu_{{X}}}^{2}}}, 
\end{align*}
and 
\begin{align*}
    \|(\mathcal{R}_{\mu_{{X}}} - \mathcal{S}_{X_k})\circ \mathcal{K}_{\mu_{{X}}}(f)\|_{L_{\mu_{X}}^2}& = \sqrt{\int_{X} \left((\mathcal{R}_{\mu_{{X}}} - \mathcal{S}_{X_k})\circ \mathcal{K}_{\mu_{{X}}}(f)\right)^2 d\mu_{X}} \\
    &= \sqrt{\frac{1}{n}\sum_{i=1}^{n}\left((\mathcal{R}_{\mu_{{X}}} - \mathcal{S}_{X_k})\circ \mathcal{K}_{\mu_{{X}}}(f)(\bx_{i})\right)^2}\\
    &= \sqrt{\frac{1}{n}\sum_{i=1}^{n}\left((\mathcal{R}_{\mu_{{X}}}\circ \mathcal{K}_{\mu_{{X}}}(f)(\bx_{i}) - \mathcal{S}_{X_k}\circ\mathcal{K}_{\mu_{{X}}}(f)(\bx_{i}))\right)^2}.
\end{align*}

Define two vectors based on the two function evaluations at $X$:
\begin{align*}
   \mathbf{F}_{1}=(\mathcal{R}_{\mu_{{X}}}\circ \mathcal{K}_{\mu_{{X}}}(f))(X), \quad \text{and}\quad
\mathbf{F}_2= (\mathcal{S}_{X_k}\circ \mathcal{K}_{\mu_{{X}}}(f))(X).
\end{align*}
Then we obtain
\begin{align*}
    \|(\mathcal{R}_{\mu_{{X}}} - \mathcal{S}_{X_k})\circ \mathcal{K}_{\mu_{{X}}}(f)\|_{L_{\mu_{X}}^2}& = \frac{1}{\sqrt{n}}\Vert \mathbf{F}_1-\mathbf{F}_2\Vert .
\end{align*}
Notice that $\mathbf{F}_1$ and $\mathbf{F}_2$ can also be written as
\begin{align*}
   \mathbf{F}_{1}=\frac{1}{n}\bK f(X), \quad \text{and}\quad
\mathbf{F}_2= \frac{1}{n}\bK_{X,X_k}\bK_{X_k,X_k}^{-1}\bK_{X_k,X}f(X).
\end{align*}
Thus, 
\begin{align*}
    \|(\mathcal{R}_{\mu_{{X}}} - \mathcal{S}_{X_k})\circ \mathcal{K}_{\mu_{{X}}}(f)\|_{L_{\mu_{X}}^2}& = \frac{1}{\sqrt{n}}\Vert \frac{1}{n}\bK f(X)-\frac{1}{n}\bK_{X,X_k}\bK_{X_k,X_k}^{-1}\bK_{X_k,X}f(X)\Vert\\
    & =\frac{1}{n^{3/2}}\Vert (\bK-\Knys)f(X)\Vert .
\end{align*}
On the other hand, 
\begin{align*}
    \|f\|_{L_{\mu_{X}}^2}& = \sqrt{\int_{X} f^2 d\mu_{X}} = \sqrt{\frac{1}{n}\sum_{i=1}^{n}f(\bx_{i})^2}= \frac{1}{\sqrt{n}}\Vert f(X)\Vert.
\end{align*}

As a result, we get 
\begin{align*}
\|(\mathcal{R}_{\mu_{X}} - \mathcal{S}_{X_k})\circ \mathcal{K}_{\mu_{{X}}}\|_{L_{\mu_{X}}^2\rightarrow L_{\mu_{{X}}}^{2}} &= \max_{f\in L_{\mu_{X}}^2,f\neq 0}\frac{\|(\mathcal{R}_{\mu_{{X}}} - \mathcal{S}_{X_k})\circ \mathcal{K}_{\mu_{{X}}}(f)\|_{L_{\mu_{X}}^2}}{\|f\|_{L_{\mu_{X}}^2}}\\
& = \max_{f\in L_{\mu_{X}}^2,f\neq 0}\frac{\Vert (\bK-\Knys)f(X)\Vert}{n\Vert f(X)\Vert}\\
&=  \max_{\mathbf{f}\in \mathbb{R}^{n},\mathbf{f}\neq 0}\frac{\Vert (\bK-\Knys)\mathbf{f}\Vert}{n\Vert \mathbf{f}\Vert} = \frac{1}{n}\Vert \bK- \Knys\Vert .
\end{align*}


Since there exists an orthogonal basis $\{f_i\}_{i=1}^{n}$ of eigenfunctions of $ \mathcal{K}_{\mu_{X}}$ in $L_{\mu_{X}}^2$ with the eigenvalues $\lambda_i$, we can express any $f \in L_{\mu_{X}^2}$ as $f = \sum_{i=1}^{n}\alpha^{(i)}f_i$. As a result, we have
\begin{align*}
    \|\mathcal{K}_{\mu_{X}}\|_{L_{\mu_{{X}}}^{2}\rightarrow \mathcal{H}_{{X}}}^2 = & \max_{f\in L_{\mu_{X}}^2,f\neq 0}\frac{\|\mathcal{K}_{\mu_{{X}}}(f)\|_{\mathcal{H}_{{X}}}^2}{\|f\|_{L_{\mu_{X}}^2}^2}\\
     & = \max_{f\in L_{\mu_{X}}^2,f\neq 0}\frac{\langle \sum_{i=1}^{n}\alpha^{(i)} \mathcal{K}_{\mu_{{X}}}(f_i), \sum_{i=1}^{n}\alpha^{(i)} \mathcal{K}_{\mu_{{X}}}(f_i) \rangle_{\mathcal{H}_{{X}}}}{\|\sum_{i=1}^{n}\alpha^{(i)}f_i\|_{L_{\mu_{{X}}}^{2}}^2}.
\end{align*}
Proposition 10.28 in \cite{wendland_scattered_2004} shows that $\{\mathcal{K}_{\mu_{{X}}}(f_i)\}$ is   orthogonal in $\mathcal{H}_{X}$:
\begin{align}
\label{eq:duality l2 and H}
\langle\mathcal{K}_{\mu_{{X}}}(f_i), \mathcal{K}_{\mu_{{X}}}(f_j)\rangle_{\mathcal{H}_{{X}}} = \langle \mathcal{R}_{\mu_{X}}\mathcal{K}_{\mu_{{X}}}(f_i), f_j\rangle_{L_{\mu_{X}}^2} = \lambda_i\langle f_i, f_j\rangle_{L_{\mu_{X}}^2}.
\end{align}
Thus we obtain
\begin{align*}
    \|\mathcal{K}_{\mu_{X}}\|_{L_{\mu_{{X}}}^{2}\rightarrow \mathcal{H}_{{X}}}^2 
    &  =\max_{f\in L_{\mu_{X}}^2,f\neq 0}\frac{\sum_{i=1}^{n}\lambda_i \vert \alpha^{(i)}\vert^2 \Vert  f_i \Vert_{L_{\mu_{X}}^2}^2}{\sum_{i=1}^{n}\vert \alpha^{(i)}\vert^2\|f_i\|_{L_{\mu_{X}}^2}^2}\\
    &= \max_{f\in L_{\mu_{X}}^2,f\neq 0}
    \sum_{i=1}^{n} \frac{\vert \alpha^{(i)}\vert^2 \Vert f_i \Vert_{L_{\mu_{X}}^2}^2}{\sum_{i=1}^{n}\vert \alpha^{(i)}\vert^2\|f_i\|_{L_{\mu_{X}}^2}^2}\lambda_i \\
    &= \lambda_1.
        \end{align*}

Since 
$$\mathcal{K}_{\mu_{{X}}}(f_i)(X) = \lambda_if_i(X) \quad \text{and} \quad \mathcal{K}_{\mu_{{X}}}(f_i)(X) = \frac{1}{n} \bK f_i(X),$$
we get
$$ \bK f_i(X) = n\lambda_if_i(X),$$ which implies that $n\lambda_i$ are the eigenvalues of the kernel matrix $\bK$ and in particular,
\begin{equation}
n\lambda_{1}=\Vert \bK\Vert.
\label{eq:lambda1}
\end{equation}

Finally, we have
\begin{align*}
    \frac{1}{n}\Vert \bK- \Knys\Vert < \sqrt{\lambda_1} C'\exp(-C''/h_{X_k})=\frac{1}{\sqrt{n}}\sqrt{\Vert \bK \Vert} C'\exp(-C''/h_{X_k}).
\end{align*}

\end{proof}

\section*{Acknowledgments}
The authors thank Zachary Frangella for sharing with us his MATLAB implementation of the randomized Nystr\"om preconditioner \cite{frangella_randomized_2021}. {The authors also would like to thank two anonymous referees for their valuable suggestions which greatly improve the presentation of the paper.}

\bibliographystyle{siam}
\bibliography{references}

\end{document}